\documentclass{article}
\usepackage[utf8]{inputenc}
\usepackage{amsmath}
\usepackage{amsfonts}
\usepackage{amssymb}
\usepackage{graphicx}
\usepackage{tikz}
\usepackage{tikz-3dplot}
\usepackage{fullpage}

\usetikzlibrary{shapes,shadows,arrows,trees}
\usetikzlibrary{trees}
\usetikzlibrary{automata,positioning}
\usepackage[hidelinks]{hyperref}

\usepackage{amsthm}
\usepackage{hyperref}
\usepackage{amsmath}
\usepackage{mathrsfs}
\usepackage[nameinlink]{cleveref}
    \crefname{ex}{Example}{Examples}
    \crefname{thm}{Theorem}{Theorems} 
    \crefname{lem}{Lemma}{Lemmas}
    \crefname{prop}{Proposition}{Propositions}
    \crefname{cor}{Corollary}{Corollaries} 
    \crefname{conj}{Conjecture}{Conjectures} 
    \crefname{defn}{Definition}{Definitions}
    \crefname{rmk}{Remark}{Remarks}

	\newtheorem{thm}{Theorem}[section]
	\newtheorem{lem}[thm]{Lemma}
	\newtheorem{prop}[thm]{Proposition}
	\newtheorem{cor}[thm]{Corollary}
	\newtheorem{conj}{Conjecture}[section]
	\theoremstyle{definition} 
	\newtheorem{defn}[thm]{Definition}
	\newtheorem{ex}[thm]{Example}
	\newtheorem{exs}[thm]{Examples}
    \newtheorem{rmk}[thm]{Remark}	
    \usepackage{authblk}
    \newcommand{\ac}{{\it 1-}$alt^c$\ }
    \newcommand{\at}{{\it 2-}$alt$\ }

\title{Classification of multistationarity for mass action networks with one-dimensional stoichiometric subspace}
\author{Casian Pantea\thanks{cpantea@math.wvu.edu} }
\author{Galyna Voitiuk\thanks{galvoit@gmail.com}}
\affil{Department of Mathematics, West Virginia University}
\date{}

\begin{document}
\maketitle

\abstract{We characterize completely the capacity for (nondegenerate) multistationarity of mass action reaction networks with one-dimensional stoichiometric subspace in terms of reaction structure. Specifically, we show that networks with two or more source complexes have the capacity for multistationarity if and only if they have both patterns $(\to, \gets)$ and $(\gets, \to)$ in some 1D projections. Moreover, we specify the classes of networks for which only degenerate multiple steady states may occur. In particular, we characterize the capacity for nondegenerate multistationarity of small networks composed of one irreversible and one reversible reaction, or two reversible reactions.

\section{Introduction}

Multistationarity (i.e. the existence or two or more compatible positive equilibrium points) underlies  switching behavior in biochemical systems, and is a key mathematical feature of systems that generate multiple outputs in response to different external signals or stimuli~\cite{angeli2004detection, craciun2006understanding, ferrell2002self}. This phenomenon plays crucial roles in cell biology, including in generating sustained oscillatory responses, remembering transitory stimuli, differentiation, or apoptosis \cite{markevich,ferrell, murray, tyson81,pomerening}.

While the question of determining which reaction networks admit multistationarity is very much open, many powerful techniques relating the structure of a reaction network with its capacity for multistationarity are known. These include tools based on the injectivity of the vector field~\cite{CracFein05, Craciun.2006ac, CracFein2006,joshi2012simplifying,wiuf2013power,Mueller.2016aa,BanPant16}, degree theory \cite{conradi2017graph,conradi2017identifying}, and inheritance, where the multistationarity of a reaction network is inherited from multistationarity of its substructures \cite{BanPant16, Banaji_20220}. This creates the possibility of ``lifting" multistationarity from idealized network motifs~\cite{alon2007network} to larger and more realistic networks.

There is already a significant amount of work on cataloguing classes of motifs by their presence or absence of multistationarity. These include all bimolecular open networks with two reactions (both reactions reversible or irreversible) \cite{JoshiShiu13}, fully open as well as  isolated sequestration networks in arbitrary number of species  and reactions \cite{joshi2015survey}, fully open networks with cyclic DSR graph ({\em CST networks} \cite{craciun2022multistationarity}), and networks with two reactions \cite{JoshiShiu17}. 

Our work adds to the catalog the class of reaction networks with one-dimensional stoichiometric subspace answering in the affirmative a conjecture posed by Joshi and Shiu:

\begin{conj}[Question 6.1 \cite{JoshiShiu17}]\label{conj:1}
A reaction network with one-dimensional stoichiometric subspace and more than one source complex has the capacity for multistationarity if and only if it has a one-species embedded subnetwork with the pattern $(\to,\gets)$, and another (possibly the same) with pattern $(\gets, \to)$.   
\end{conj}

Additionally, we characterize which networks with one-dimensional stoichiometric subspace can have {\em nondegenerate} equilibria. The classification of networks with nondegenerate equilibria  was required in our proof of Conjecture \ref{conj:1}, but it is important on its own: for example inheritance results only apply for nondegenerate equilibria. Our classification of nondegeneracy extends and generalizes results from the literature \cite{JoshiShiu17, ShiuWolff19, Lin2022}. In particular we prove a conjecture by Shiu and de Wolff regarding characterization of nondegenerate multistationarity for networks with one irreversible and one reversible reaction.  

Conjecture \ref{conj:1} is known to hold for special cases. In particular, the conjecture is true for networks with two reactions (Joshi and Shiu \cite{JoshiShiu17}) and under some technical additional assumptions (Lin, Tang and Zhang \cite{Lin2022}). These are discussed in detail in Section \ref{sec:disc}. 

Our proof depends in a critical way on two inheritance theorems for nondegenerate equilibria. The first (Joshi and Shiu \cite[Theorem 3.1]{JoshiShiu13}, Banaji and Pantea \cite[Theorem 1]{BanPant18}) studies persistence of multistationarity if reactions are added to the network, and the second  (Banaji et al. \cite[Theorem 1]{Banaji_2020}) if species are added to the network. These results are discussed at the end of section \ref{sec:multistat}. 

The paper is structured as follows: general reaction networks terminology is introduced in section \ref{sec:prelim}; section \ref{sec:1Drn} sets notation for networks with one-dimensional stoichiometric subspace and defines classes of networks needed to state our main result (Theorem \ref{thm:mainext}). The main result, examples, and connection with literature are presented in section \ref{sec:main}, and section \ref{sec:proofs} contains the proof. 


\section{Background}\label{sec:prelim}
\subsection{Reaction networks and kinetics}\label{sec:rnk}
Throughout the paper $\{e_1,\ldots, e_n\}$ denotes the standard basis of ${\mathbb R}^n$ and $\mathbb R_{>0}^n$, $\mathbb R_{\ge 0}^n$  denote the subsets of $\mathbb R^n$ containing vectors with positive and non-negative entries respectively. The set of $n$-dimensional vectors with nonnegative integer entries is denoted by ${\mathbb Z}_{\ge 0}^n$. We review standard terminology on reaction networks.

A {\em reaction} on a list of {\em species} $X=(X_1,\ldots X_n)$ has the general form
\begin{equation}\label{eq:genReact}
a_1X_1+a_2X_2+\ldots+a_nX_n\to b_1X_1+b_2X_2+\ldots+b_nX_n,
\end{equation}
where and $a=(a_1,\ldots, a_n)$ and $b=(b_1,\ldots, b_n)$ are vectors in ${\mathbb Z}_{\ge 0}^n$. Entries $a_i$ and $b_i$ are called {\em stoichiometric coefficients} of species $X_i$ in $a$ and $b$. The {\em source complex} of (\ref{eq:genReact}) is
$a\cdot X := a_1X_1+a_2X_2+\ldots+a_nX_n$, and $b\cdot X$ is called the {\em product complex} of the reaction. $b-a$  is called the {\em   reaction vector} of (\ref{eq:genReact}).

\smallskip

A {\em reaction network} $\cal N$ on a list of species $X=(X_1,\ldots, X_n)$
is a finite list of reactions on $X$. We require that source and product complexes differ for each reaction, and that no reaction is listed multiple times. 

Once an order is chosen for reactions in $\cal N$, their reaction vectors form the columns of the {\em stoichiometric matrix} $\Gamma\in{\mathbb Z}^{n\times r}$.  The image of $\Gamma$ is called the {\em stoichiometric subspace} of the network, and its rank is called the {\em rank of $\cal N$} \cite{Banaji_2020}. For the remainder of the paper we refer to networks with one-dimensional stoichiometric subspace as {\em rank 1 reaction networks}. 

The vector of {\em concentrations} of $X_1,\ldots, X_n$ is denoted by $x\in\mathbb R^n_{\ge 0}$. In deterministic spatially homogeneous models the time evolution of concentration is commonly modeled using {\em mass action kinetics}, where the {\em reaction rate} of $a\to b$ is given by 
$\kappa x^a:=\kappa x_1^{a_1} \ldots x_n^{a_n}$. Here $\kappa$ is a positive constant that depends on the reaction, called the {\em rate constant} of the reaction. 

A reaction network $\cal N$ with an assignment of rate constants $\kappa=(\kappa_1,\ldots, \kappa_m)$ is called a {\em mass action system} and is denoted by $({\cal N},\kappa)$. Under mass action the time evolution of the concentration vector $x$ is governed by the ODE system of $({\cal N},\kappa):$
\begin{equation}\label{eq:genKin}
\frac{\text{d}x}{\text{d}\tau} = \Gamma r(x(\tau)).
\end{equation}
 $\tau$ denotes the time variable, and
$r(x)=(r_1(x),\ldots r_m(x))^t$ is the vector of reaction rates (i.e. the {\em rate vector}). Note that $x(\tau)$ stays nonnegative for any $\tau\ge 0$ (see for example \cite{Volpert72}).

\subsection{Compatibility classes and multistationarity}\label{sec:multistat}

Integrating \eqref{eq:genKin} with respect to time we have
$$x(\tau)=x(0)+\Gamma\int_{0}^T r(x(s))\text{d}s,$$
in other words the solutions of \eqref{eq:genKin} are constrained to affine sets of
the form 
$$(\eta+\text{im }\Gamma)\cap{\mathbb R}_{\ge 0},$$ where $\eta\in{\mathbb R}_{\ge 0}^n$. These are called {\em compatibility classes}. 

A {\em positive equilibrium}, or {\em positive steady state} of a mass action system $({\cal N},\kappa)$ is a point $x^*\in\mathbb R^n_{>0}$ such that $$\Gamma r(x^*)=0.$$ 

{\em Multistationarity} refers to the existence of two or more positive equilibria. With solutions being constrained to compatibility classes, the notion of multistationarity is at the level of each compatibility class:

\begin{defn}\em{(Capacity for MPE).} We say that a reaction network $\cal N$ {\em has the capacity for multiple positive equilibria (MPE)} if there exists an assignment of rate constants $\kappa$ such that the reaction system $({\cal N},\kappa)$ has two or more positive equilibria belonging to the  same compatibility class of $\cal N$.
\end{defn}

\begin{ex}\label{ex:nonMPE}
Consider the reaction network \footnote{ $0=0X_1+0X_2$ denotes the {\em zero complex}, which may be interpreted as a placeholder for the exterior of the environment where the reactions take place, or for species that we do not include in our model. For example (\ref{eq:ex1intro}) may represent the biochemically realistic network $X_1+X_1\to 2X_3$, $2X_3\to X_1+X_2$, where $X_3$ is so abundant that it is considered constant for all practical purposes, and not included in the model.}
\begin{equation}
\label{eq:ex1intro}
{\cal N}_1=\{X_1+X_2\to 2X_1+2X_2,\quad  X_1+X_2{\to} 0,\quad 0\to X_1+X_2\}. 
\end{equation}
Letting $\kappa_1$ and $\kappa_2$ denote the rate constants of the two reactions, under mass action equilibrium points satisfy the equation
$$
\Gamma r(x)=
\begin{pmatrix}
1 &-1 &1\\
1 &-1 &1
\end{pmatrix}
\begin{pmatrix}
\kappa_1x_1x_2\\
\kappa_2x_1x_2\\
\kappa_3
\end{pmatrix}
=0,
$$
or $(\kappa_1-\kappa_2)x_1x_2+\kappa_3=0$; this curve is called {\em the steady state manifold} of mass action system $({\cal N}_1,\kappa)$. 
The stoichiometric subspace of ${\cal N}_1$ is one-dimensional and compatibility classes are cosets of $\text{im}\Gamma=\text{span}((1,1)^t)$ intersected with the positive quadrant, i.e. line segments of the form $x_2-x_1=T$, Figure \ref{fig:exintro}$(a)$. Positive equilibria in the compatibility class $x_2-x_1=T$ are obtained by solving the system 
$$(\kappa_1-\kappa_2)x_1x_2+\kappa_3=0,\quad x_2-x_1=T.$$ 
This has no positive solutions for $\kappa_1>\kappa_2$ and exactly one solution when $\kappa_1<\kappa_2$. The stoichiometric class $x_2-x_1=T$ contains no equilibrium points or exactly one steady state 
$(1/2)(-T+\sqrt{T^2-\frac{4\kappa_3}{\kappa_1-\kappa_2}}, T+\sqrt{T^2-\frac{4\kappa_3}{\kappa_1-\kappa_2}})$.
Therefore ${\cal N}_1$ does not have the capacity for MPE. 
\end{ex}

\begin{ex}\label{ex:ex2intro}
The reaction network
\begin{equation}
\label{eq:ex2intro}
{\cal N}_2=\{2X_1+X_2{\to} 3X_1,\quad  X_1{\to} X_2\} 
\end{equation}
is a subnetwork of a bistable network for modeling apoptosis \cite{ho2010bistability} and
is one of the simplest networks with bistability \cite{joshi13, BanPant18}.
Assigning rate constants $\kappa_1=\kappa_2=1$ to reactions above, equilibria are computed by solving the equation
$$
\Gamma r(x)=
\begin{pmatrix}
1 &-1\\
-1 &1
\end{pmatrix}
\begin{pmatrix}
x_1^2x_2\\
x_1
\end{pmatrix}
=0,
$$
and the steady state manifold of $({\cal N}_2,\kappa)$ in the positive orthant is $x_1x_2=1$. The stoichiometric subspace is spanned by $(1,-1)^t$ and compatibility have the form $x_1+x_2=T$. Positive equilibria in the compatibility class $x_1+x_2=T$ are obtained by solving the system 
$$x_1x_2=1,\quad x_1+x_2=T.$$ 
This gives $x_1^2-Tx_1+1=0$; there are no positive equilibria if $T<2$, one positive equilibrium $(x_1,x_2)=(1,1)$ if $T=2$ and two positive equilibria $(\frac{T+\sqrt{T^2-4}}{2},\frac{T-\sqrt{T^2-4}}{2})$ and $(\frac{T-\sqrt{T^2-4}}{2},\frac{T+\sqrt{T^2-4}}{2})$ if $T>2$ (Figure \ref{fig:exintro}$(b)$). ${\cal N}_2$ has the capacity for MPE. 
\begin{figure}
\begin{center}
\begin{tikzpicture}[scale=1.5, scale = 1.5]
\begin{scope}
\draw[->, color=black!60] (0,0) -- (1.5,0);
\draw[->, color=black!60] (0,0) -- (0,1.5);
\draw[thick,scale=0.5,domain=0.037:2.5,smooth,variable=\y]  plot ({\y},{.1/\y});
\draw[color=black!60] (0,.4) -- (.9,1.3);
\draw[color=black!60] (0,.8) -- (.5,1.3);
\draw[color=black!60] (0,0) -- (1.3,1.3);
\draw[color=black!60] (.4,0) -- (1.3,.9);
\draw[color=black!60] (.8,0) -- (1.3,.5);
\fill (.155,.155) circle (.4pt);
\fill (.455,.055) circle (.4pt);
\fill (.83,.03) circle (.4pt);
\fill (.045,.845) circle (.4pt);
\fill (.06,.46) circle (.4pt);
\node[color=black!60] (a) at (1.5,.13) {$x_1$};
\node[color=black!60] (a) at (.13, 1.5) {$x_2$};
\node at (.75,-.3) {$(a)$};
\end{scope}
\begin{scope}[xshift=60]
\draw[->, color=black!60] (0,0) -- (1.5,0);
\draw[->, color=black!60] (0,0) -- (0,1.5);
\draw[thick,scale=0.5,domain=0.037:2.6,smooth,variable=\y]  plot ({\y},{.1/\y});
\draw[color=black!60] (0,.15) -- (.15,0);
\draw[color=black!60] (0,.31) -- (.31,0);
\draw[color=black!60] (0,.7) -- (.7,0);
\draw[color=black!60] (0,1) -- (1,0);
\draw[color=black!60] (0,1.3) -- (1.3,0);
\fill (.155,.155) circle (.4pt);
\fill (.055,.65) circle (.4pt);
\fill (.04,.96) circle (.4pt);
\fill (.66,.04) circle (.4pt);
\fill (.98,.03) circle (.4pt);
\fill (1.3-0.019,0.019) circle (.4pt);
\fill (0.019,1.3-0.019) circle (.4pt);
\node (a) at (.24,.24) {\small $A$};
\node (a) at (.16,.96) {\small $B$};
\node (a) at (.98,.14) {\small $C$};
\node[color=black!60] (a) at (1.5,.13) {$x_1$};
\node[color=black!60] (a) at (.13, 1.5) {$x_2$};
\node at (.75,-.3) {$(b)$};
\end{scope}
\begin{scope}[xshift=120]
\draw[->, color=black!60] (0,0) -- (1.5,0);
\draw[->, color=black!60] (0,0) -- (0,1.5);
\draw[color=black!60] (0,.4) -- (.9,1.3);
\draw[color=black!60] (0,.8) -- (.5,1.3);
\draw[thick] (0,0) -- (1.3,1.3);
\draw[color=black!60] (.4,0) -- (1.3,.9);
\draw[color=black!60] (.8,0) -- (1.3,.5);
\node[color=black!60] (a) at (1.5,.13) {$x_1$};
\node[color=black!60] (a) at (.13, 1.5) {$x_2$};
\node at (.75,-.3) {$(c)$};
\end{scope}
\end{tikzpicture}
\end{center}
\caption{Various scenarios for multistationarity. Steady state manifolds are drawn in bold, and compatibility classes with thin lines.   $(a)$ Reaction network (\ref{eq:ex1intro}) does not have the capacity for MPE. Each compatibility class contains exactly one nondegenerate equilibrium. $(b)$ Reaction network (\ref{eq:ex2intro}) has the capacity for MPNE. Compatibility classes may contain no equilibria, one degenerate equilibrium (point $A$), or two nondegenerate equilibria. $(c)$
Reaction network (\ref{eq:ex3intro}) has the capacity for MPE, but not the capacity for MPNE. When equilibria exist, a whole compatibility class is made out of equilibria. }
\label{fig:exintro}
\end{figure}
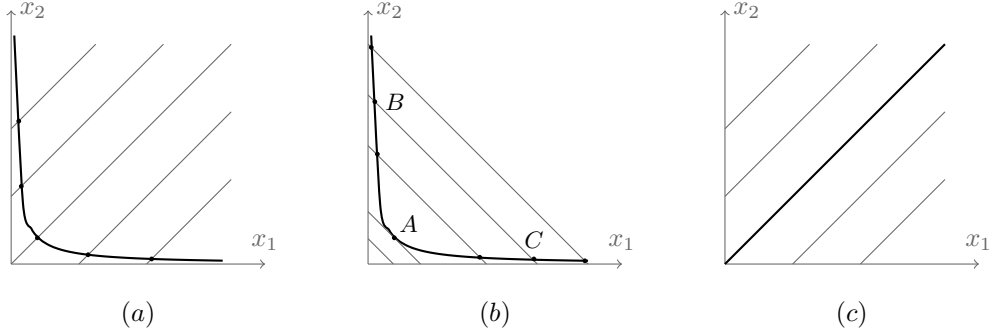
\end{ex}

\noindent {\bf Nondegenerate equilibria}. At equilibrium point $A$ in Figure \ref{fig:exintro}$(b)$ the steady state manifold and the compatibility class intersect tangentially, while at equilibrium points $B$ and $C$ the intersection is transversal. This distinction is an important one: $A$ is called a {\em degenerate equilibrium}, while $B$ and $C$ are called {\em nondegenerate equilibria}. To be precise, a positive equilibrium point $x^*$ is called nondegenerate if (see for example \cite[Definition 4]{Craciun.2006ac})
\begin{equation}\label{eq:degen}
\ker (\Gamma Dr(x)|_{x=x^*})\cap \text{im }\Gamma=\{0\}.    
\end{equation}
Equivalently, a positive equilibrium $x^*$ is nondegenerate if the Jacobian of the vector field projected on compatibility classes ({\em reduced Jacobian} \cite[Appendix A.]{BanPant16} or {\em core determinant} \cite{Helton.2009aa}) is nonzero. Conveniently for calculations, the reduced Jacobian is the sum of all $k\times k$ principal minors of the Jacobian matrix $\Gamma Dr(x)_{x=x^*}$, where $k=\text{rank }\Gamma$ is the dimension of the stoichiometric subspace \cite{BanPant16}. 

\begin{defn}\em{(Capacity for MPNE).} We say that a reaction network $\cal N$ {\em has the capacity for multiple positive nondegenerate equilibria (MPNE)} if there exist an assignment of rate constants $\kappa$ such that the mass action system $({\cal N},\kappa)$ has two or more positive nondegenerate equilibria belonging to the  same compatibility class of $\cal N$.
\end{defn}

The Jacobian matrix of network ${\cal N}_1$ in example \ref{ex:ex2intro} taken with rate constants $\kappa_1=\kappa_2=1$ is 
$$\Gamma Dr(x)=
\begin{pmatrix}
2x_1x_2-1 &x_1^2\\
-2x_1x_2+1 &-x_1^2
\end{pmatrix}
$$
and its reduced Jacobian is equal to 
$-x_1^2+2x_1x_2-1$. 
At equilibrium points $x^*$ this expression is equal to $-(x^*_1)^2+1$, and it follows that $(1,1)$ is a degenerate equilibrium, whereas all other equilibrium points are nondegenerate. Reaction network (\ref{eq:ex2intro}) has the capacity for MPNE. 

Some reaction networks exhibit pathological multistationarity behaviour, in that equilibria, if they exists, are all degenerate.

\begin{ex}\label{ex:ex3intro}
Consider the reaction network
\begin{equation}\label{eq:ex3intro}
{\cal N}_3=\{X_1+2X_2\to 2X_1+3X_2,\quad 2X_1+X_2\to X_1\}
\end{equation}
\end{ex}
The steady state manifold in the positive quadrant is $\kappa_1x_2-\kappa_2x_1=0$. The stoichiometric subspace of ${\cal N}_3$ is  spanned by $(1,1)^t$ and compatibility classes are of the form $x_2-x_1=T$. 
Positive equilibria in the compatibility class $x_2-x_1=T$ obey the equation 
$(\kappa_1-\kappa_2)x_1+\kappa_1T=0$, which has  positive solutions if and only if $\kappa_1=\kappa_2$ and $T=x_2-x_1=0$. In other words, all positive points in the compatibility class $x_2-x_1=0$ are equilibria of ${\cal N}_3$, and no other equilibria exist (Figure \ref{fig:exintro}$(c)$). The reduced Jacobian  $\kappa_1x_2^2-\kappa_2x_1^2$ is zero at equilibrium points and therefore all equilibria are degenerate: ${\cal N}_3$ has the capacity for MPE, but does not have the capacity for MPNE. 

We conclude this section with a useful observation that allows us to remove species with corresponding zero coordinates in each reaction vector; their concentrations are constant functions of time. 
Let $\cal N$ be a reaction network with stoichiometric matrix $\Gamma$ and suppose row $i_0$ of $\Gamma$ is zero. If $a$ is a complex of $\cal N$ we let $\bar a$ denote the complex obtained by removing species $i_0$ from $a$, and let 
$$\bar{\cal N}=\{\bar a\to \bar b|a\to b\in{\cal N}\}.$$
($\bar{\cal N}$ is obtained from ${\cal N}$ by removing species $i_0$ from all reactions).
If $\kappa_k$ denotes the rate constant of  $a\cdot X\to b\cdot X$ we  define $\bar{\kappa}_k=\kappa_kC^{a_{i_0}}$ to be the rate constant of $\bar a\to \bar b$.  
We have $\kappa_kx^a=\bar{\kappa_k}{\bar a}^{\bar x}$, so that $r(x)={\bar r}({\bar x})$ (${\bar r}$ denotes the rate vector of $(\bar{\cal N},\bar \kappa)$).

\begin{prop}\label{prop:nonzerov}
Suppose $\cal N$ is a reaction network with row $i_0$ of its stoichiometric matrix equal to zero. For any $C>0$, $y=(y_1,\ldots,y_{i_0-1},C, y_{i_0+1}, \ldots, y_n)\in{\mathbb R}^n_{>0}$ is a (degenerate/nondegenerate) equilibrium point of $({\cal N},\kappa)$ if and only if $\bar y=(y_1,\ldots, y_{i_0-1},y_{i_0+1},\ldots, y_n)$ is a (degenerate/nondegenerate) equilibrium point of $(\bar{\cal N},\bar\kappa)$ where $\bar \kappa$ is defined above.
\end{prop}
\begin{proof}
We may assume that $i_0=1$. By separating the first row in $\Gamma$ we write in block matrix form 
\begin{equation*}
\Gamma r(y)=\Gamma {\bar r}({\bar y})=
\begin{pmatrix}
0\\
\bar\Gamma
\end{pmatrix}{\bar r}({\bar y})
=\begin{pmatrix}
0\\ 
{\bar\Gamma}{\bar r}({\bar y})
\end{pmatrix},
\end{equation*}
so $\Gamma r(y)=0$ if and only if ${\bar\Gamma}{\bar r}({\bar y})=0$. 

For the nondegeneracy part we use (\ref{eq:degen}). We can write
$Dr(x)=D_{r(x)}R^tD_{1/x}$, where $R$ denotes the {\em reactant matrix} having the source complexes of $\cal N$ as columns, and $D_\eta$ denotes the diagonal matrix with entries of vector $\eta$ on the diagonal. If  $v=(v_1,\ldots, v_n)\in\text{im }\Gamma$ is a nonzero vector, then  $v_1=0$ and $\bar v=(v_2,\ldots, v_n)\in \text{im }{\bar\Gamma}$, $\bar v\neq 0$. Conversely, if ${\bar v}\in\text{im }{\bar\Gamma}$ is nonzero, then $v=(0,{\bar v})\in \text{im }{\Gamma}$ is nonzero. Using block matrix forms we have 
\begin{eqnarray*}
\Gamma Dr(x)|_{x=y}v=
\begin{pmatrix}
0\\
\bar\Gamma
\end{pmatrix}
D_{{\bar r}({\bar x}})
\begin{pmatrix}
R_1 &{\bar R}
\end{pmatrix}
\begin{pmatrix}
1/x_1 &0\\
0 &D_{1/{\bar x}}
\end{pmatrix}
\begin{pmatrix}
0\\
{\bar v}
\end{pmatrix}
=
\begin{pmatrix}
0\\
{\bar\Gamma}D_{{\bar r}({\bar x})}{\bar R}{\bar v}
\end{pmatrix}=
\begin{pmatrix}
0\\
{\bar\Gamma}D{{\bar r}({\bar x})}_{{\bar x}=
{\bar y}}{\bar v}
\end{pmatrix}
\end{eqnarray*}
so $\Gamma Dr(x)|_{x=y}v=0$ if and only if
${\bar\Gamma}D{{\bar r}({\bar x})}_{{\bar x}=
{\bar y}}{\bar v}=0$, 
therefore $y$ is nondegenerate if and only if ${\bar y}$ is.
\end{proof}

Noting that $y$, $z$ are compatible equilibria of ${\cal N}$ if and only of 
$\bar y$, $\bar z$ are compatible equilibria of $\bar{\cal N}$, we have the following 

\begin{cor}\label{cor:nonzerovMPNE}
$\cal N$ has the capacity for MPE (MPNE) if and only if $\bar {\cal N}$ does. 
\end{cor}


\medskip

\noindent {\bf Inheritance of multistationarity}. Our proof is based on the following two results.
\begin{thm}[Theorem 3.1 \cite{JoshiShiu13}, Theorem 1 \cite{BanPant18}] \label{thm:inheritance}
Let ${\cal N}$ be a reaction network on $n$ species and ${\cal N}'$ a reaction network on the same species obtained by adding reactions to ${\cal N}$. Suppose ${\cal N}'$ has the same rank as $\cal N$. If ${\cal N}$ has the capacity for MPNE, then so does ${\cal N}'$.
\end{thm}

\begin{thm}[Theorem 1 \cite{Banaji_2020}]\label{thm:inheritance2} Let $\cal N$ be a reaction network and let ${\cal N}'$ be obtained from $\cal N$ by adding a new species in any of the complexes of $\cal N$, with any stoichiometric coefficient.  Suppose ${\cal N}'$ has the same rank as $\cal N$. If $\cal N$ has the capacity for MPNE then so does ${\cal N}'$.
\end{thm}

Theorems \ref{thm:inheritance}  and \ref{thm:inheritance2} are only two of the known network  modifications that preserve nondegenerate multistationarity \cite{BanPant18}; see also \cite{feliu2013simplifying, cappelletti2020addition}. Recent work of Banaji and others showed that these same network modifications also preserve oscillatory behavior \cite{BANAJI2018191, Banaji_2020, Banaji_20220, banaji221}.  

\subsection{Reaction networks, embedded graphs, and arrow diagrams}

For the purpose of this paper, it is useful to introduce the following definition, even though it comes at the price of a slight abuse of terminology.
\begin{defn}{\em (Reaction network).}\label{def:rn}
A \em reaction network $\cal N$ on $n$ species is a finite set of {\em reactions}
$$a_k\to a'_k,\ k\in\{1,\ldots, r\}$$ 
where $a_k\in {\mathbb Z}_{\ge 0}^n$, $a'_k\in {\mathbb Z}^n$ and $a_k\neq a'_k$ for all $k\in\{1,\ldots, r\}$,
\end{defn}

With terminology already introduced, $a_1,\ldots, a_k$ are called {\em source complexes} of $\cal N$; the $n\times r$ matrix with columns $a'_1-a_1,\ldots, a'_r-a_r$ is called the {\em stoichiometric matrix} of the network; $\text{im }\Gamma$ is called its {\em stoichiometric subspace}; and if $\eta\in{\mathbb R}_{>0}^n$ then $(\eta+\text{im }\Gamma)\cap{\mathbb R}^n_{\ge 0}$, and $(\eta+\text{im }\Gamma)\cap{\mathbb R}^n_{>0}$ are called the {\em compatibility class}, respectively {\em positive compatibility class of $\eta$.}

Each reaction $a_k\to a'_k$ is associated a {\em rate constant} $\kappa_k>0$. Letting 
$$r(x)=(\kappa_1x^{a_1},\ldots, \kappa_rx^{a_1})$$
the ODE system (\ref{eq:genKin}) models the {\em mass action system} $({\cal N},\kappa)$. Positive equilibrium points, degeneracy/nondegeneracy, and the capacity for MPE or MPNE are defined as in section \ref{sec:multistat}. 

\begin{defn}{\em (Subnetwork)} A subnetwork of a reaction network $\cal N$ is a reaction network given by a nonempty subset of the reactions of $\cal N$.
\end{defn}

\begin{rmk}
Definition \ref{def:rn} extends the usual notion of reaction network allowing product complexes to have negative coordinates. Some dynamical properties of mass action systems change under the new definition;  for instance the positive orthant may not be forward invariant, e.g.  reaction $0\to -1$ on one species. However, on questions related to positive equilibria there is no drawback to this extension. A motivating feature  of  Definition \ref{def:rn} is that replacing all reaction vectors of $\cal N$ by their negatives  yields a new reaction network $-\cal N$, and $(-\cal N,\kappa)$ has the same (nondegenerate/degenerate) positive equilibria as $({\cal N,\kappa})$. This allows us to reduce the number of cases in some of the proofs.
\end{rmk}

\begin{rmk}
Our proof can be easily extended to allow source complexes and reaction vectors to have {\em real} entries, i.e. the vector field is a {\em signomial} vector function reaction. Reaction rates in this case are called {\em power law} \cite{savageau1988introduction}, and are used to model a variety of biological phenomena including genetic circuits and developmental systems \cite{savageau1974comparison, savageau1979allometric}.
A number of results on various dynamical aspects of mass-action kinetics (like multistationarity \cite{BanPant16}, persistence and global stability \cite{Craciun.2013ab, Crac19}) apply more generally for power-law kinetics.
\end{rmk}

A reaction network $\cal N$ can be depicted naturally as arrows between points of ${\mathbb Z}_{\ge 0}^n$, see Figure \ref{fig:egsd}$(a)$ for an example. Following \cite{Crac19} we call this the {\em embedded graph of $\cal N$}. Note that the embedded graph of a reaction network $\cal N$ identifies $\cal N$ uniquely up to permutation of species indices. 

\footnote{This is not true in the original definition of an euclidean embedded graph \cite{Crac19}; in that work the embedded graph is defined starting from an ODE system rather than a network, like we do here.}

The {\em arrow diagram} of $\cal N$ is obtained by replacing all arrows from $a_k$ to $a'_k$ in the embedded graph by arrows starting at $a_k$ in the direction of $sign(a'_k-a_k)\in\{-1,0,1\}^n$. Vectors from the same source complex going in the same direction are drawn as one. The length of arrows in the arrow diagram is not important; see Figure \ref{fig:egsd}$(b)$.

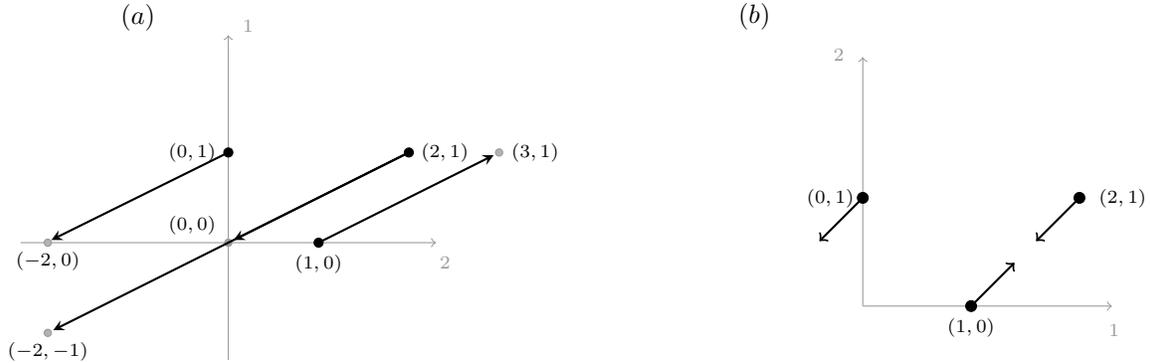
\begin{figure}[h]
\centering
\begin{tikzpicture}[scale=1.2]
\begin{scope}
\draw [->, color=black!40](-2.3,0)--(2.3,0);
\draw [->, color=black!40](0,-1.3)--(0,2.3);
\draw[color=black!40] (.22,2.4) node {\scriptsize $1$};
\draw[color=black!40] (2.4,-.22) node {\scriptsize $2$};
\draw[fill=black] (0,1) circle[ radius=0.05];
\draw[fill=black] (2,1) circle[ radius=0.05];
\draw[fill=black] (1,0) circle[ radius=0.05];
\draw[color=black!40, fill=black!40] (0,0) circle[ radius=0.04];
\draw[color=black!40,, fill=black!30] (-2,0) circle[ radius=0.04];
\draw[color=black!40,, fill=black!30] (3,1) circle[ radius=0.04];
\draw[color=black!40,, fill=black!30] (-2,-1) circle[ radius=0.04];
\draw [-stealth,thick](0,1)--(-1.96,.03);
\draw [-stealth,thick](2,1)--(.06,.03);
\draw [-stealth,thick](2,1)--(-1.94,-.97);
\draw [-stealth,thick](1,0)--(2.94,.97);
\draw (-.4,1) node {\scriptsize $(0,1)$};
\draw (-.4,.2) node {\scriptsize $(0,0)$};
\draw (-2,-.2) node {\scriptsize $(-2,0)$};
\draw (-2,-1.2) node {\scriptsize $(-2,-1)$};
\draw (2.4,1) node {\scriptsize $(2,1)$};
\draw (3.4,1) node {\scriptsize $(3,1)$};
\draw (1,-.24) node {\scriptsize $(1,0)$};
\draw (-1,2.5) node {$(a)$};
\end{scope}
\begin{scope}[xshift=200, yshift=-20, scale=1.2]
\draw [->, color=black!40](0,0)--(2.3,0);
\draw [->, color=black!40](0,0)--(0,2.3);
\draw[color=black!40] (2.32,-.22) node {\scriptsize $1$};
\draw[color=black!40] (-.22,2.32) node {\scriptsize $2$};
\draw[fill=black] (0,1) circle[ radius=0.05];
\draw[fill=black] (2,1) circle[ radius=0.05];
\draw[fill=black] (1,0) circle[ radius=0.05];
\draw [->,thick](0,1)--(-.4,0.6);
\draw [->,thick](2,1)--(1.6,0.6);
\draw [->,thick](1,0)--(1.4,0.4);
\draw (-.3,1) node {\scriptsize $(0,1)$};
\draw (2.4,1) node {\scriptsize $(2,1)$};
\draw (1,-.2) node {\scriptsize $(1,0)$};
\draw (-1,2.68) node {$(b)$};
\end{scope}
\end{tikzpicture}
\caption{$(a)$ Embedded graph and $(b)$ arrow diagram of reaction network $\{(0,1)\to (-2,0),\ (2,1)\to (0,0),\ (2,1)\to (-2,-1),\ (1,0)\to (3,1)$)\}.}\label{fig:egsd}
\end{figure}
\subsection{Network projections}\label{sec:proj}
\begin{defn}{\em (Network projection)}\footnote{Projections of networks are called {\em embedded networks} by other authors \cite{JoshiShiu17}.}\label{def:proj}
Let $\cal N$ be a reaction network and let $I$ be a nonempty subset of $\{1,\ldots,n\}$. The {\em projection of $\cal N$ on $I$}, denoted by ${\cal N}_I$ is the reaction network on species in $I$ with reactions $p_I(a_k)\to p_I(a'_k)$ where $p_I$ denotes the orthogonal projection on coordinates ${\mathbb Z}^I$. If $p_I(a_k)=p_I(a'_k)$ then that reaction is discarded from ${\cal N}_I$. Duplicate reactions are also discarded from ${\cal N}_I$.  
\end{defn}

In other words, ${\cal N}_I$ is obtained by removing coordinates $i\notin I$ from all complexes in $\cal N$ and then discarding any trivial reactions where the source complex and the product complex is the same. We note that for rank 1 reaction networks no projection of a reaction can be trivial if $\Gamma$ contains only nonzero rows. Moreover, while two different reactions in $\cal N$ may have the same projection in ${\cal N}_I$, we list each reaction in ${\cal N}_I$ only once. The embedded graph of ${\cal N}_I$ is the orthogonal projection on coordinates $I$ of the embedded graph of $\cal N$, discarding trivial reactions.  

If $|I|=k$ we call  ${\cal N}_I$ a {\em kD projection} of $\cal N$. Arrow diagrams of 2D projections can be depicted in two ways on a 2D cartesian grid, depending on which 1D projection is drawn horizontally (Figure \ref{fig:proj}(i),(ii)) Two 2D arrow diagrams that are reflections of each other about the diagonal of the positive orthant represent the same network, up to permuting coordinates.

\begin{ex}\label{ex:proj}
Let
$${\cal N}=\{2X_1+X_3\to X_1+X_2+3X_3,\;X_1+X_2\to 2X_2+2X_3,\;X_2+2X_3\to X_1\}.
$$

Dropping $X$ from notation the 2D projections of $\cal N$ on $\{1,2\}$ and $\{3\}$ are
\begin{eqnarray*}
{\cal N}_{\{1,2\}}&=&\{
(2,0)\to (1,1),\;
(1,1)\to (0,2),\;
(0,1) \to (1,0)\}\\
{\cal N}_{\{3\}}&=&\{
1\to 3,\;
0\to 2,\;
2 \to 0\}
\end{eqnarray*}

Figure \ref{fig:proj} illustrates arrow diagrams of some 2D and 1D projections of $\cal N$.
\end{ex}
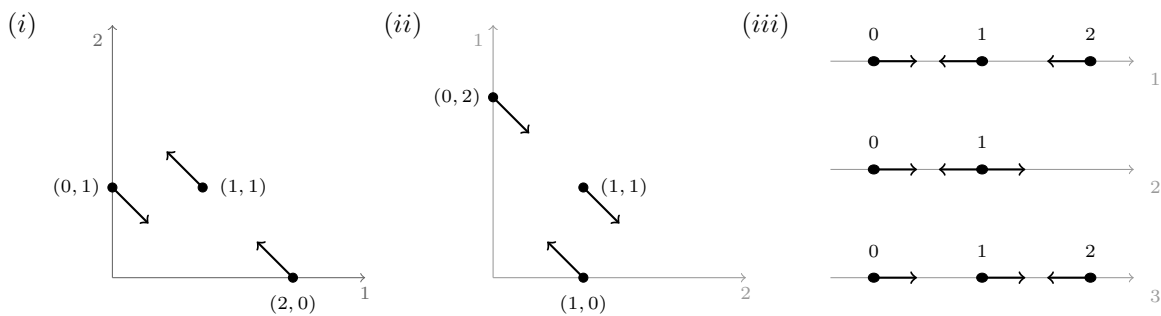
\begin{figure}[h]
\centering
\begin{tikzpicture}[scale=1.2]
\begin{scope}
\draw [->, color=black!60](0,0)--(2.8,0) node[anchor= north]{\scriptsize $1$};
\draw [->, color=black!60](0,0)--(0,2.8)
node[anchor=north east]{\scriptsize $2$};
\draw[fill=black] (2,0) circle[ radius=0.05];
\draw[fill=black] (1,1) circle[ radius=0.05];
\draw[fill=black] (0,1) circle[ radius=0.05];
\draw [->, thick](2,0)--(1.6,.4);
\draw [->, thick](1,1)--(.6,1.4);
\draw [->, thick](0,1)--(0.4,.6);
\draw (-.4,1) node {\scriptsize $(0,1)$};
\draw (1.45,1) node {\scriptsize $(1,1)$};
\draw (2,-.3) node {\scriptsize $(2,0)$};
\draw (-1,2.8) node {$(i)$};
\end{scope}
\begin{scope}[xshift=120]
\draw [->, color=black!40](0,0)--(2.8,0) node[anchor=north]{\scriptsize $2$};
\draw [->, color=black!40](0,0)--(0,2.8)
node[anchor=north east]{\scriptsize $1$};
\draw[fill=black] (0,2) circle[ radius=0.05];
\draw[fill=black] (1,1) circle[ radius=0.05];
\draw[fill=black] (1,0) circle[ radius=0.05];
\draw [->, thick](0,2)--(.4,1.6);
\draw [->, thick](1,1)--(1.4,.6);
\draw [->, thick](1,0)--(.6,0.4);
\draw (-.4,2) node {\scriptsize $(0,2)$};
\draw (1.45,1) node {\scriptsize $(1,1)$};
\draw (1,-.3) node {\scriptsize $(1,0)$};
\draw (-1,2.8) node {$(ii)$};
\end{scope}
\begin{scope}[xshift=240, xscale=1.2]
\draw [->, color=black!40](-.4,2.4)--(2.4,2.4);
\draw [->, color=black!40](-.4,1.2)--(2.4,1.2);
\draw [->, color=black!40](-.4,0)--(2.4,0);
\draw[color=black!40] (2.6,2.2) node {\scriptsize $1$};
\draw[color=black!40] (2.6,1) node {\scriptsize $2$};
\draw[color=black!40] (2.6,-.2) node {\scriptsize $3$};
\draw (0,2.7) node {\scriptsize 0};
\draw (1,2.7) node {\scriptsize 1};
\draw (2,2.7) node {\scriptsize 2};
\draw (0,1.5) node {\scriptsize 0};
\draw (1,1.5) node {\scriptsize 1};
\draw (0,.3) node {\scriptsize 0};
\draw (1,.3) node {\scriptsize 1};
\draw (2,.3) node {\scriptsize 2};
\draw[fill=black] (0,0) circle[ radius=0.05];
\draw[fill=black] (1,0) circle[ radius=0.05];
\draw[fill=black] (2,0) circle[ radius=0.05];
\draw[fill=black] (0,1.2) circle[ radius=0.05];
\draw[fill=black] (1,1.2) circle[ radius=0.05];
\draw[fill=black] (0,2.4) circle[ radius=0.05];
\draw[fill=black] (1,2.4) circle[ radius=0.05];
\draw[fill=black] (2,2.4) circle[ radius=0.05];
\draw [->, thick](0,2.4)--(.4,2.4);
\draw [->, thick](1,2.4)--(.6,2.4);
\draw [->, thick](2,2.4)--(1.6,2.4);
\draw [->, thick](0,1.2)--(.4,1.2);
\draw [->, thick](1,1.2)--(.6,1.2);
\draw [->, thick](1,1.2)--(1.4,1.2);
\draw [->, thick](0,0)--(.4,0);
\draw [->, thick](1,0)--(1.4,0);
\draw [->, thick](2,0)--(1.6,0);
\draw (-1,2.8) node {$(iii)$};
\end{scope}
\end{tikzpicture}
\caption{Arrow diagrams of projections of $\cal N$ in example \ref{ex:proj}. $(i),(ii)$: both figures represent ${\cal N}_{\{1,2\}}$; the two pictures are reflections of each other about the diagonal. $(c)$: 1D projections ${\cal N}_{\{1\}}$, ${\cal N}_{\{2\}}$, ${\cal N}_{\{3\}}$.}
\label{fig:proj}
\end{figure}

\section{Rank 1 reaction networks}\label{sec:1Drn}
For rank 1 networks $\cal N$, the inheritance theorems \ref{thm:inheritance} and \ref{thm:inheritance2} state simply that
{\em
\begin{itemize}
\item[(1)] if a subnetwork of $\cal N$ has the capacity for MPNE, then so does ${\cal N}$, and
\item[(2)] if a projection of $\cal N$ has the capacity for MPNE, then so does ${\cal N}$\footnote{(1) and (2) are equivalent to saying that if an {\em embedded subnetwork} \cite{JoshiShiu13} of $\cal N$ has the capacity for MPNE, then so does $\cal N$.}.

\end{itemize}
}
 These two key facts allow us to piece together structural requirements for degeneracy from (subnetworks of) smaller dimensional projections. Facts (1) and (2) are uniquely important in this work. To avoid invoking them excessevily, we will without further reference use the fact that 
 \smallskip 
 \begin{center}
 {\em if a subnetwork of a projection of $\cal N$ has the capacity for MPNE, then so does $\cal N$.}
 \end{center}

\subsection{Preliminaries}\label{sec:notation}
Let $\cal N$ be a rank 1 reaction network. If a source complex of $\cal N$ is involved in two or more reactions in the same direction then they can be merged into a single reaction. The sum of their rate constants is assigned as a rate constant to the new reaction. In this way the two mass-action reaction systems follow the same set of ODEs and in particular steady states and their nondegeneracy/degeneracy stay the same.\footnote{This construction is a particular case of {\em dynamically equivalent} mass action networks \cite{Craciun.2008aa}.} This allows us to assume that {\em each source complex is associated with either one reaction, or with two reactions of opposite directions.}

Next we set notation which will be used throughout the paper. In some proofs we override this common notation to simplify presentation.

\medskip

\noindent {\bf Notation.}
\begin{itemize}
\item[1] $\cal N$ denotes a rank 1 reaction network on $n\ge 1$ species with $m\ge 1$ source complexes. 

\item[2] The index sets for species and source complexes are denoted by $S$ and $R$ respectively: 
$$S=\{1,\ldots, n\} \text{ and } R=\{1,\ldots, m\}.$$
\item[3] $v\in {\mathbb R}^n$ spans the stoichiometric subspace of ${\cal N}$. The sign pattern of $v$ is encoded in the sets 
$$S^+=\{i\in S|v_i>0\},\quad S^-=\{i\in S|v_i<0\}.$$
By Proposition \ref{prop:nonzerov} we may assume that {\em $v$ does not have zero coordinates.} Note however that the statement of our main theorem is independent of this assumption. 
\item[4.1] If $a_k$ is source complex for only one reaction, that reaction is denoted by
$a_k \to a_k^{\prime},$
and we let $a_k^{\prime}-a_{k}={\lambda}_k v.$ 
The rate constant of $a_k \to a_k^{\prime}$ is denoted by $\kappa_k$.

\item[4.2] If $a_k$ is the source complex of two reactions, these are  denoted by
$a_k \to a_k^{\prime +},\ a_k \to a_k^{\prime -}$
and we let
$a_k^{\prime +}-a_k=\lambda_k^+v,\ a_k^{\prime -}-a_k=\lambda_k^-v$
where $\lambda_k^+>0$ and $\lambda_k^-<0.$
The rate constant of $a_k \to a_k^{\prime +}$ is denoted by $\kappa_k^+$ and that of reaction $a_k \to a_k^{\prime -}$ by $\kappa_k^-$.

We say that two reactions {\em have the same direction} if their corresponding $\lambda$ values have the same sign, and that they {\em have opposite directions} otherwise.

\item[4.3] The following sets indicate the source complexes that have reactions in each of the two directions:
\begin{eqnarray*}
&R^+&=\{k\in R| a_k\text{ is source complex for two reactions or } \lambda_k>0\}\\
&R^-&=\{k\in R| a_k\text{ is source complex for two reactions or } \lambda_k<0\}\\
&R^0&=\{k\in R| a_k\text{ is source complex for two reactions}\}=R^+\cap R^-.
\end{eqnarray*}
%
\item[5] The mass-action ODE system for $(\cal N,\kappa)$ is given by
$$\frac{d x}{d \tau}=F(x), \text{ where }F(x)=\left(\sum_{k\in R\setminus R^0}\lambda_k \kappa_kx(\tau)^{a_k}+\sum_{k\in R^0}(\lambda_k^+\kappa_k^++\lambda_k^-\kappa_k^-) x(\tau)^{a_k}\right)v.$$
We use $\tau$ to denote the time variable. Variable $t$ is used as follows:

\item[6] The positive stoichiometric class of $\eta\in {\mathbb R}_{>0}^n$ is parameterized by  $\eta+tv$. The bounds for $t$ are computed from $\eta_i+tv_i>0$ for all $i\in S$, i.e.
$t\in (\alpha^\eta,\beta^\eta)=(-\min_{i\in S^+} {\eta_i}/{v_i}, -\max_{i\in S^-}{\eta_i}/{v_i})$. 
\item[7] We let 
$$f(\eta,\kappa;t)=\sum_{k\in R\setminus R^0}\lambda_k\kappa_k(\eta+tv)^{a_k}
+\sum_{k\in R^0}(\lambda_k^+\kappa_k^++\lambda_k^-\kappa_k^-)(\eta+tv)^{a_k}.
$$  
$f$ is as a function of $t$ that depends on the parameters $\eta, \kappa$. The domain of $f(\eta,\kappa;t)$ is $(\alpha^\eta,\beta^\eta)$.
\end{itemize}

\medskip

\noindent {\bf Positive equilibria and $f(\eta,\kappa;t)$.} Clearly, $\eta+t_0v\in{\mathbb R}^n_{>0}$ is a steady state of $(\cal N,\kappa)$ if and only if
\begin{equation}\label{eq:sst_0}
f(\eta,\kappa;t_0)=0
\end{equation}
Moreover, we have 
\begin{prop}\label{prop:DSSf}
A steady state $\eta+t_0v$ of the mass action system $({\cal N},\kappa)$ is degenerate if and only if $f'(\eta,\kappa;t_0)=0$.
\end{prop}

\begin{proof}
We have $f(\eta,\kappa;t)v=F(\eta+tv)$ and $f'(\eta,\kappa;t)v=DF|_{\eta+tv} v$. Then $f'(\eta,\kappa;t_0)=0$ if and only if $v\in\ker(DF|_{\eta+t_0v}).$
\end{proof}

This yields the following useful observation.

\begin{cor}\label{cor:infeq}
If a rank 1 reaction network has an infinite number of positive equilibria within the same compatibility class, then these are all degenerate. 
\end{cor}
\begin{proof}
If the polynomial $f(\eta,\kappa;t)$ has an infinite number of roots it must be identically zero. 
\end{proof}

In some of proofs below $f$ is written as a product. The following lemma is useful in that case.

\begin{lem}\label{prop:DSSh}
Suppose $\eta+t_0v$ is a steady state of the reaction system $({\cal N},\kappa)$ and that $f(\eta,\kappa;t)=p(t)q(t)$ where $p(t_0)\ne0$. Then $\eta+t_0v$ is degenerate if and only if $q'(t_0)=0.$
\end{lem}
\begin{proof}
 $f'(\eta,\kappa;t_0)=p'(t_0)q(t_0)+p(t_0)q'(t_0)=p(t_0)q'(t_0)$, so $f'(\eta,\kappa;t_0)=0$ if and only if $q'(t_0)=0$.
\end{proof}

\medskip
\noindent {\bf The essential projection.} Let $\cal N$ denote a rank 1 reaction network on $n$ species, and let 
$$I=\{i\in S|\ {\cal N}_{\{i\}} \text{ has two or more source complexes and } v_{i}\neq 0\}.$$
Then ${\cal N}_I$ is called the {\em essential projection} of $\cal N$. If $I=S$ then we call $\cal N$ an {\em essential reaction network}.

In other words, the essential projection of $\cal N$ removes the coordinates corresponding to $v_i=0$ and those where all source complexes project to to the same point. The next simple calculation together with Proposition \ref{prop:nonzerov} show that the equilibria of the network obtained this way are the same and they have the same degenerate/nondegenerate type. Let $J=S\setminus I$. For each $j\in J$ we have $a_{kj}=a_{lj}:=\zeta_j$ for all $k\in R$. If  $x\in{\mathbb R}^S$ we let $x_I$ denote the projection of $x$ on $I$. 

\begin{prop}\label{prop:essential} Let $\cal N$ be a rank 1 reaction network and let ${\cal N}_I$ denote its essential projection. 
\begin{itemize} 
\item[(1)] if $\eta+tv$ is a (nondegenerate) positive equilibrium of $({\cal N},\kappa)$ then $\eta_I+t_0v_I$ is  a (nondegenerate) positive equilibrium  of $({\cal N}_I,\kappa)$.

\item[(2)] If $\eta'+t_0v_I$ is a (nondegenerate) positive steady state of $({\cal N}_I,\kappa)$ we choose for every $j\in J$ $\eta_j$ large enough so that $\eta_j+t_0v_j>0$ and we let $\eta=(\eta_I,\eta_J)\in {\mathbb Z}^S$. Then
$\eta+t_0v$ is a (nondegenerate) positive steady state of $({\cal N},\kappa)$.
\end{itemize}
\end{prop}
\begin{proof}
1. We have
\begin{eqnarray*}
f(\eta,\kappa;t)&=&\prod_{j\in S\setminus I}(\eta_j+tv_j)^{\zeta_j}\left(
\sum_{k\in R\setminus R^0}\lambda_k\kappa_k(\eta_I+tv_I)^{(a_k)_I}
+\sum_{k\in R^0}(\lambda_k^+\kappa_k^+ +\lambda_k^-\kappa_k^-) (\eta_I+tv_I)^{{(a_k)}_I}\right)\\
&=& f(\eta_I,\kappa;t)
\end{eqnarray*}
and the conclusion follows from proposition \ref{prop:nonzerov} and lemma \ref{prop:DSSh}.

2. Similar to 1.
\end{proof}

\begin{cor}
A reaction network $\cal N$ has the capacity for MPE (MPNE) if and only if its essential projection ${\cal N}_I$ has the capacity for MPE (MPNE). 
\end{cor}
\begin{proof}
If $\cal N$ has the capacity for MPE (MPNE) then ${\cal N}_I$ does by corollary \ref{cor:nonzerovMPNE} and Proposition \ref{prop:essential}(1). If ${\cal N}_I$ has the capacity for MPE (MPNE) we let $\eta'+t_1v_I$, $\eta'+t_2v_I$ be distinct positive equilibria (nondegenerate equilibria) of ${\cal N}_I$ and for each $j\in J$ choose $\eta_j$ large enough so that $\eta_j+t_1v_j>0$, $\eta_j+t_2v_j>0$. Then 
$\eta+t_1v$, $\eta+t_2v$ distinct are positive equilibria (nondegenerate equilibria) of $({\cal N},\kappa)$ by corollary \ref{cor:nonzerovMPNE} and Proposition \ref{prop:essential}(2).
\end{proof}

\subsection{Special classes of networks}
Following \cite{JoshiShiu17} we have the following definitions.
\begin{defn}\label{defn:patterns}
\begin{itemize} 
\item[1.] We say that a reaction network $\cal N$ on one species {\em contains the $(\to, \gets)$ pattern} if it contains reactions $a\to a'$, $b\to b'$ with $a<a'$, $b'<b$ and $a<b$. ${\cal N}$ {\em has the $(\gets, \to)$ pattern} if it contains reactions $a\to a'$, $b\to b'$ with $a'<a$, $b'>b$ and $a<b$.

\item[2.] We say that a rank 1 reaction network $\cal N$ contains the $(\to, \gets)$ (or $(\gets,\to)$) pattern if there exists a 1D projection of $\cal N$ that has this pattern.
\end{itemize}
\end{defn}

In the same way as in Definition \ref{defn:patterns} arrow diagrams of 1D networks are represented by a a sequence of arrows $\to$, $\gets$ or bidirectional arrows $\leftarrow\!\!\bullet\!\!\rightarrow$ (reactions of opposite directions starting at the same source complex), for example $(\to, \leftarrow\!\!\bullet\!\!\rightarrow, \to, \gets)$. Commas separate reactions with distinct source complexes, and complexes increase to the right.

\begin{defn}{\em (Class {\it 1-}$alt^c$: 1-alt complete networks).} A rank 1 reaction network $\cal N$ is called {\em 1-alternating complete} if it contains both $(\to, \gets)$ and $(\gets, \to)$ patterns. The two 1D projections may be on the same coordinate. The class of 1-alternating complete networks is denoted by \ac. We write ${\cal N}\in$\ac and refer to $\cal N$ as {\em 1-alt complete}.
\end{defn}

The direction (increasing or decreasing) of the projection of $a_k\to a'_k$ on $\{i\}$ is given by the sign of $\lambda_kv_i$. 



\begin{defn}\label{def:2alt}
A rank 1 reaction network $\cal N$ that has a 1D projection containing both arrow patterns 
$$(\to,\gets,\to),\quad (\gets,\to,\gets)$$ on some 1D projection is called {\em 2-alternating}. We write ${\cal N}\in$\at.
\end{defn}

\begin{rmk}\label{rmk:1alt2alt}
Clearly, 2-alternating networks are 1-alt complete; moreover, it is easy to restricted to networks on one species \ac is the (disjoint) union of \at and the set of networks with arrow diagram $(\leftarrow\!\!\bullet\!\!\rightarrow, \leftarrow\!\!\bullet\!\!\rightarrow)$.
\end{rmk}
\begin{defn}\label{defn:zpattern}
Let ${\cal N}=\{a\to a', b\to b'\}$ be a rank 1 reaction network on two species. We say that ${\cal N}$ is a {\em zigzag} and if
\begin{itemize}
    \item[1.] $a\to a'$ and $b\to b'$ have opposite directions and
    \item[2.] $v_1v_2(b_1-a_1)(b_2-a_2)<0$.

The {\em slope} of zigzag ${\cal N}$ is defined as the slope of $b-a$, i.e. $\frac{b_2-a_2}{b_1-a_1}$. 
\end{itemize}
\end{defn}

Part 2 of Definition \ref{defn:zpattern} says that the slopes of $v$ and $b-a$ are nonzero and of different signs, in other words the  arrow diagram of ${\cal N}$ is one of those in Figure \ref{fig:zigzag}. The rectangles in the figure (``boxes" in terminology from \cite{JoshiShiu17}) are not degenerate to line segments.

\begin{figure}[h]
   \centering
\begin{tikzpicture}[scale=0.5]
\draw [color=black!60](0,2)--(3,0)
(6,2)--(9,0)
(12,0)--(15,2)
(18,0)--(21,2);
\draw [color=black!40] (0,0) -- (0,2) -- (3,2) -- (3,0) -- cycle;
\draw [color=black!40] (6,0) -- (6,2) -- (9,2) -- (9,0) -- cycle;
\draw [color=black!40]  (12,0) -- (12,2) -- (15,2) -- (15,0) -- cycle;
\draw [color=black!40] (18,0) -- (18,2) -- (21,2) -- (21,0) -- cycle;
\draw [->, thick](0,2)--(1,3);
\draw [->, thick](3,0)--(2,-1);
\draw [->, thick](6,2)--(5,1);
\draw [->, thick](9,0)--(10,1);
\draw [->, thick](12,0)--(11,1);
\draw [->, thick](15,2)--(16,1);
\draw [->, thick](18,0)--(19,-1);
\draw [->, thick](21,2)--(20,3);
\draw[ fill=black] (3,0)circle[radius=0.1];
\draw[ fill=black] (0,2)circle[radius=0.1];
\draw[ fill=black] (9,0)circle[radius=0.1];
\draw[ fill=black] (6,2)circle[radius=0.1];
\draw[ fill=black] (12,0)circle[radius=0.1];
\draw[ fill=black] (15,2)circle[radius=0.1];
\draw[ fill=black] (18,0)circle[radius=0.1];
\draw[ fill=black] (21,2)circle[radius=0.1];
\end{tikzpicture}
\caption{Arrow diagrams of zigzags.}
\label{fig:zigzag}
\end{figure}
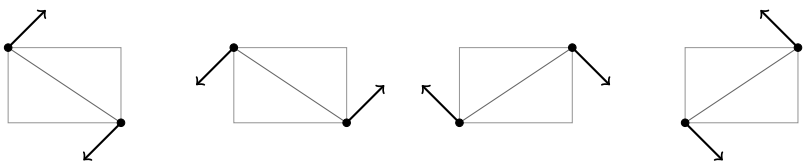


\begin{defn}\label{defn:znetwork} {\em (Class ${\cal Z}$)}.
A rank 1 reaction network is called a {\em zigzag network} if it has a 2D projection containing a zigzag. We write ${\cal N}\in {\cal Z}$.
\end{defn}

Zigzag networks are 1-alt complete, and in fact ``minimally" so: only two reactions are required to produce both $(\to,\gets)$ and $(\gets, \to)$ patterns. Other networks, like the one in Figure \ref{fig:egsd}, require three reactions to produce these patterns. Yet other networks require four reactions for this purpose, see proposition \ref{prop:s2nz}.

The following lemma will be useful later.

\begin{lem}\label{cor:31}
Let $\mathcal{N}=\{a\to a^{\prime}, b\to b'\}$ be 1-alt complete reaction network, and suppose ${\cal N}_{\{i,j\}}\in{\cal Z}$ is the only 2D projection of ${\cal N}$ that contains a zigzag. Then $a_s=b_s$ for all $s\in \{1,\ldots, n\}\setminus\{i,j\}$.
\end{lem}
\begin{proof}
We have $v_iv_j(a_i-b_i)(a_j-b_j)<0$. Assume that there exists $s\notin\{i,j\}$ such that $a_s\neq b_s$. Then  $v_iv_s(a_i-b_i)(a_s-b_s)<0$ or $v_jv_s(a_j-b_j)(a_s-b_s)<0$ and therefore ${\cal N}_{\{i,s\}}\in {\cal Z}$ or ${\cal N}_{\{j,s\}}\in {\cal Z}$, which is a contradiction. 
\end{proof}


\begin{exs}\label{ex:altz}
\text{ }

{\em 1.} The reaction network in Figure \ref{fig:egsd} contains the $(\to, \gets)$ pattern in both its 1D projection, and contains the $(\gets,\to)$ pattern in its projection on $\{1\}$ . This network is 2-alternating.

\smallskip
{\em 2.} The reaction network $\cal N$ in Example \ref{ex:proj}
contains the $(\to,\gets)$ pattern on all of its 1D projections, but it does not contain the $(\gets,\to)$ pattern. This network is not 1-alt complete.

\smallskip
{\em 3.} Let 
${\cal N}=\{(0,1,0)\to(1,2,-1),\ (1,0,1)\to(0,-1,2)\}$. This network contains the $(\to,\gets)$ pattern on its $\{1\}$ projection, and the $(\gets,\to)$ on its $\{2\}$ and $\{3\}$ projections. $\cal N$ is 1-alt complete. Moreover, $\cal N$ is a zigzag network.  ${\cal N}_{\{1,2\}}$ is a zigzag of slope -1 and ${\cal N}_{\{1,3\}}$ is a zigzag of slope 1. 
\end{exs}

\begin{figure}[h]
\centering
\begin{tikzpicture}[scale=1]
\begin{scope}
\draw [->, color=black!40](0,0)--(2.8,0) node[anchor= north]{\scriptsize $1$};
\draw [->, color=black!40](0,0)--(0,2.8)
node[anchor=north east]{\scriptsize $2$};
\draw[fill=black] (1.5,0) circle[ radius=0.05];
\draw[fill=black] (0,1.5) circle[ radius=0.05];
\draw [->, thick](1.5,0)--(1.1,-.4);
\draw [->, thick](0,1.5)--(.4,1.9);
\draw (-.4,1.5) node {\scriptsize $(0,1)$};
\draw (1.6,-.3) node {\scriptsize $(1,0)$};
\draw (0,1.5)--(1.5,0);
\draw (-1,2.8) node {$(i)$};
\end{scope}
\begin{scope}[xshift=240]
\draw [->, color=black!40](0,0)--(2.8,0) node[anchor=north]{\scriptsize $1$};
\draw [->, color=black!40](0,0)--(0,2.8)
node[anchor=north east]{\scriptsize $3$};
\draw[fill=black] (1.5,1.5) circle[ radius=0.05];
\draw[fill=black] (0,0) circle[ radius=0.05];
\draw [->, thick](1.5,1.5)--(1.1,1.9);
\draw [->, thick](0,0)--(.4,-0.4);
\draw (-.4,0) node {\scriptsize $(0,0)$};
\draw (1.9,1.5) node {\scriptsize $(1,1)$};
\draw (0,0)--(1.5,1.5);
\draw (-1,2.8) node {$(ii)$};
\end{scope}
\end{tikzpicture}
\caption{
Arrow diagrams of 2D projections of $\cal N$ in Example \ref{ex:altz}:
${\cal N}_{\{1,2\}}$ is a zigzag of slope -1, 
${\cal N}_{\{1,3\}}\in {\cal Z}$.}

\label{fig:altz}
\end{figure}
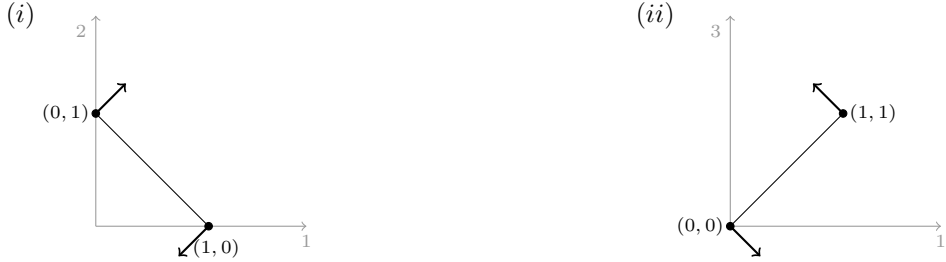

\medskip

Next we define four classes of reaction networks with that are central in the classification of  multistationarity for rank 1 reaction networks:  our main result shows that these classes encompass all networks with capacity for MPE, but no capacity for MPNE.  

\begin{defn}{\em (Class ${\cal S}_1)$: one-source networks).}
A rank 1 reaction network $\cal N$ is called an {\em ${\cal S}_1$ network} if $\cal N$ contains exactly one source complex and two reactions of opposite directions. We write ${\cal N}\in{\cal S}_1$.
\end{defn}

\begin{defn}{\em (Class ${\cal S}_2^z$: two-source zigzag networks).}
A rank 1 reaction network $\cal N$ is called an {\em ${\cal S}_2^z$ network} if it has two species, contains exactly two source complexes $a,b$ and has a zigzag of slope -1. 
\end{defn}

Networks of class ${\cal S}_2^z$ have the arrow diagram  in Figure \ref{fig:s2z} up to reflections about the diagonal (i.e permuting coordinates) and up to reversing directions of all arrows (i.e. replacing $\cal N$ by -$\cal N$). Any ${\cal S}_2^z$ network contains the arrow diagram in Figure \ref{fig:s2z}(iii).

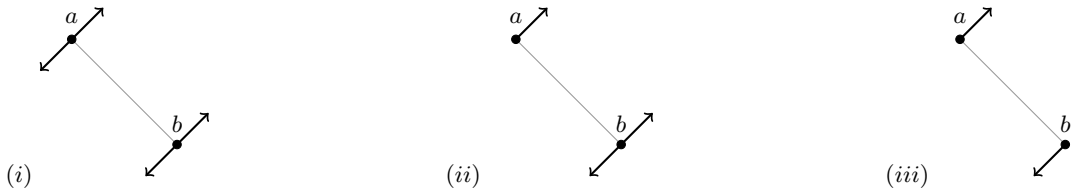
\begin{figure}[h]
\centering
\begin{tikzpicture}[scale=1.4]
\begin{scope}
\draw[color=black!40] (1,2)--(2,1);
\draw [<->,thick](.7,1.7)--(1.3,2.3);
\draw [<->, thick](1.7,.7)--(2.3,1.3);
\draw[fill=black] (1,2) circle[ radius=0.04];
\draw[fill=black] (2,1) circle[ radius=0.04];
\draw (1,2.2) node {\small $a$};
\draw (2,1.2) node {\small $b$};
\draw (.5,.7) node {\small $(i)$};
\end{scope}
\begin{scope}[xshift=120]
\draw[color=black!40] (1,2)--(2,1);
\draw [->,thick](1,2)--(1.3,2.3);
\draw [<->, thick](1.7,.7)--(2.3,1.3);
\draw[fill=black] (2,1) circle[ radius=0.04];
\draw[fill=black] (1,2) circle[ radius=0.04];
\draw (1,2.2) node {\small $a$};
\draw (2,1.2) node {\small $b$};
\draw (.5,.7) node {\small $(ii)$};
\end{scope}
\begin{scope}[xshift=240]
\draw[color=black!40] (1,2)--(2,1);
\draw [->,thick](1,2)--(1.3,2.3);
\draw [->, thick](2,1)--(1.7,.7);
\draw[fill=black] (1,2) circle[ radius=0.04];
\draw[fill=black] (2,1) circle[ radius=0.04];
\draw (1,2.2) node {\small $a$};
\draw (2,1.2) node {\small $b$};
\draw (.5,.7) node {\small $(iii)$};
\end{scope}
\end{tikzpicture}
\caption{Class ${\cal S}_2^z$: networks on two species with two source complexes whose arrow diagrams contain zigzags of slope -1. The diagrams were drawn up to reflection about the diagonal axis (i.e. permuting two species indices) and up to reversing all arrows (replacing $\cal N$ with $-{\cal N}$).}\label{fig:s2z}
\end{figure}

\begin{defn}{\em (Class $\cal L$: line networks}\label{defn:class1})
A rank 1 reaction network $\cal N$ is said to be {\em an $\cal L$ network} if it has two species, at least three source complexes, and
there exists $\delta\in\mathbb{R}_{>0}^2$ such that

\begin{equation}\label{eq:Lcases}(a_{k1}, a_{k2})=\delta+p_k(1,-1) \text{ for all } k\in R 
\text{ or } (a_{k1}, a_{k2})=\delta+p_k(-1,1) \text{ for all } k\in R
\end{equation}
where $p_k$ are pairwise distinct, $p_k<0$ for all $k\in R^+\setminus R^0$, 
$p_k=0$ for $k\in R^0$, and
$p_l>0$ for all $l\in R^-\setminus R^0$.
We write ${\cal N}\in {\cal L}.$
\end{defn}

The two statements in (\ref{eq:Lcases}) are the same after permuting coordinates. Up to this permutation, $\cal L$ networks have arrow diagrams in Figure \ref{fig:Lnew}{\it(i),(ii)}. Their source complexes lie on a line of slope -1 and their reactions are separated according to their direction. For networks of class $\cal L$, $R^0$ contains zero or one elements; if there exists a source complex $a_k$ with reactions in both directions, then $a_k=\delta$. 

\begin{rmk}\label{rmk:L}
${\cal L}$ networks are those with all source complexes on a line of slope -1, with at least three source complexes, 1-alt complete, but not 2-alternating. This easy observation will be useful later on.
\end{rmk}

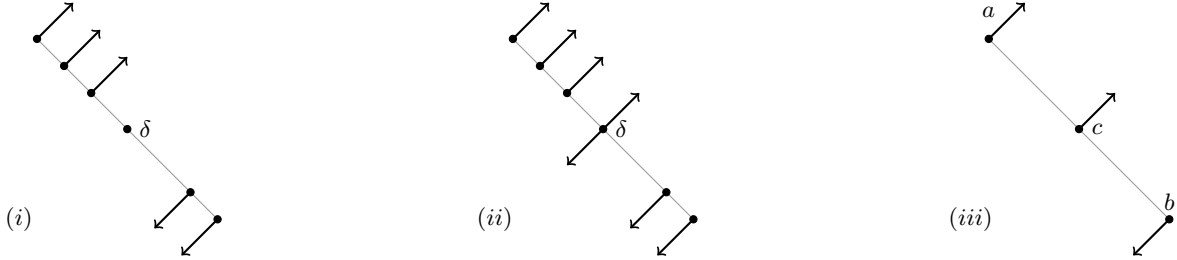
\begin{figure}[h]
\centering
\begin{tikzpicture}[scale=1.2]
\begin{scope}
\draw[color=black!40] (1,3)--(3,1);
\draw [->,thick](1,3)--(1.4,3.4);
\draw [->,thick](1.3,2.7)--(1.7,3.1);
\draw [->,thick](1.6,2.4)--(2,2.8);
\draw [->, thick](3,1)--(2.6,.6);
\draw [->, thick](2.7,1.3)--(2.3,.9);
\draw[fill=black] (1,3) circle[radius=0.04];
\draw[fill=black] (2,2) circle[radius=0.04];
\draw[fill=black] (1.3,2.7) circle[radius=0.04];
\draw[fill=black] (1.6,2.4) circle[radius=0.04];
\draw[fill=black] (3,1) circle[radius=0.04];
\draw[fill=black] (2.7,1.3) circle[radius=0.04];
\draw (.8,1) node {\small $(i)$};
\draw (2.2,2) node {\small $\delta$};
\end{scope}
\begin{scope}[xshift=150]
\draw[color=black!40] (1,3)--(3,1);
\draw [->,thick](1,3)--(1.4,3.4);
\draw [->,thick](1.3,2.7)--(1.7,3.1);
\draw [->,thick](1.6,2.4)--(2,2.8);
\draw [->, thick](3,1)--(2.6,.6);
\draw [<->, thick](1.6,1.6)--(2.4,2.4);
\draw [->, thick](2.7,1.3)--(2.3,.9);
\draw[fill=black] (1,3) circle[radius=0.04];
\draw[fill=black] (1.3,2.7) circle[radius=0.04];
\draw[fill=black] (1.6,2.4) circle[radius=0.04];
\draw[fill=black] (3,1) circle[radius=0.04];
\draw[fill=black] (2.7,1.3) circle[radius=0.04];
\draw[fill=black] (2,2) circle[radius=0.04];
\draw (.8,1) node {\small $(ii)$};
\draw (2.2,2) node {\small $\delta$};
\end{scope}
\begin{scope}[xshift=300]
\draw[color=black!40] (1,3)--(3,1);
\draw [->,thick](1,3)--(1.4,3.4);
\draw [->, thick](3,1)--(2.6,.6);
\draw [->, thick](2,2)--(2.4,2.4);
\draw[fill=black] (1,3) circle[radius=0.04];
\draw[fill=black] (3,1) circle[radius=0.04];
\draw[fill=black] (2,2) circle[radius=0.04];
\draw (1,3.3) node {\small $a$};
\draw (3,1.2) node {\small $b$};
\draw (2.2,2) node {\small $c$};
\draw (.8,1) node {\small $(iii)$};
\end{scope}
\end{tikzpicture}
\caption{Class ${\cal L}.$ {\it(i),(ii)}: arrow diagrams of 2D projections, up to reflection about the diagonal axis. {\it (iii):} any network in $\cal L$ contains a subnetwork $\{a\to a', b\to b', c\to c'\}$ with the given arrow diagram, up to reflection about the diagonal or reversing all arrows.}\label{fig:Lnew}
\end{figure}

\smallskip

\begin{defn}{\em (Class ${\cal S}_2^{nz}$: two-source non-zigzag networks).}
A rank 1 reaction network $\cal N$ is said to be {\em a ${\cal S}_2^{nz}$ network} if it is essential, contains exactly two source complexes and ${\cal N}\in$\ac$\setminus {\cal Z}$. 
We write ${\cal N}\in {\cal S}_2^{nz}.$
\end{defn}

If ${\cal N}\in{\cal S}_2^{nz}$ then $\cal N\notin Z$ and for any $i,j\in S$ we have
$v_iv_j(b_i-a_j)(b_j-a_j)\ge 0$. Since ${\cal N}$ is essential, this inequality is strict. Therefore ${\cal S}_2^{nz}$ consist of all essential networks with exactly two  source complexes $a$ and $b$, each having  reactions in both directions, such that
that $b-a$ has the same sign pattern as $v$ or $-v$:
$$b-a\in{\cal Q}(v) \text{ or } b-a\in{\cal Q}(-v)$$
where for vectors $v$ with nonzero coordinates $Q(v)=\{u\in {\mathbb R}^n| u_i v_i>0\}$. 
1D projections of ${\cal S}_2^{nz}$ networks have the arrow diagram $(\leftarrow\!\!\bullet\!\!\rightarrow, \leftarrow\!\!\bullet\!\!\rightarrow)$. Arrow diagrams of 2D projections of ${\cal S}_2^{nz}$ networks are given in Figure \ref{fig:S2nz}, up to permuting $i$ and $j$. Note that all $k$D projections of ${\cal S}_2^{nz}$ networks are ${\cal S}_2^{nz}$ networks themselves for $k\ge 2$.  

\begin{figure}[h]
\centering
\begin{tikzpicture}[scale=0.8]
\begin{scope}
    \draw [->, thick](3,4)--(3.5,4.5) ;
    \draw [->, thick](3,4)--(2.5,3.5) ;
    \draw[ color=black!50] (1,1)--(1,4)--(3,4)--(3,1)--(1,1);
    \draw [->, thick](1,1)--(1.5,1.5) ;
    \draw [->, thick](1,1)--(0.5,0.5) ;
    \draw[fill=black] (1,1) circle[ radius=0.06];
    \draw[fill=black] (3,4) circle[ radius=0.06];
    \draw (0.5,0.3) node {\scriptsize $(c)$};
    \draw (1.5,1.7) node {\scriptsize $(d)$};
    \draw (3.5,4.7) node {\scriptsize $(a)$};
    \draw (2.5,3.3) node {\scriptsize $(b)$};
\end{scope}
\begin{scope}[xshift=350]
    \draw [->, thick](1,4)--(1.5,3.5) ;
    \draw [->, thick](1,4)--(0.5,4.5) ;
    \draw [->, thick](3,1)--(3.5,0.5) ;
    \draw [->, thick](3,1)--(2.5,1.5) ;
    \draw[fill=black] (3,1) circle[ radius=0.06];
    \draw[fill=black] (1,4) circle[ radius=0.06];
    \draw[ color=black!50]  (1,1)--(1,4)--(3,4)--(3,1)--(1,1);
    \draw (3.5,0.3) node {\scriptsize $(c)$};
    \draw (2.5,1.7) node {\scriptsize $(d)$};
    \draw (1.5,3.3) node {\scriptsize $(b)$};
    \draw (0.5,4.7) node {\scriptsize $(a)$};
\end{scope}    

\end{tikzpicture}
\caption{Class ${\cal S}_2^{nz}$: four reactions $a\to a', b\to b', c\to c', d\to d'$;   2D projections have one of the two arrow diagrams in the figure, up to reflection about the diagonal. Reactions are labelled with their source complex.}\label{fig:S2nz}

\end{figure}
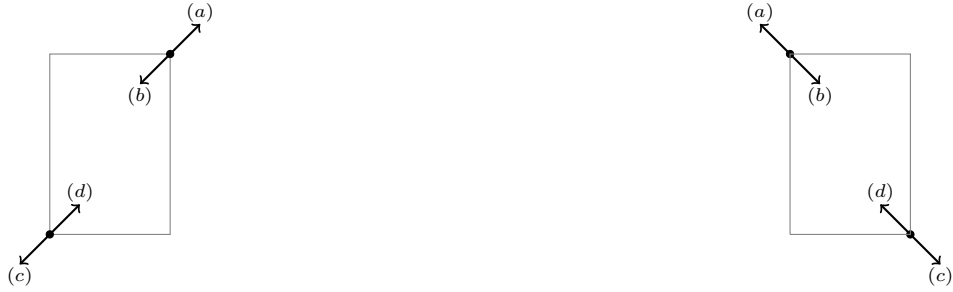

The following is a useful alternative characterization of ${\cal S}_2^{nz}$ networks:
\begin{prop}\label{prop:s2nz}
Let ${\cal N}$ be 1-alt complete. Then ${\cal N}\in{\cal S}_2^{nz}$ if and only if 
no subnetwork of $\cal N$ composed of three reactions is 1-alt complete. 
\end{prop}
\begin{proof} The statement is equivalent if we replace $\cal N$ by its essential projection, so we may assume that $\cal N$ is an essential network.
Clearly, if ${\cal N}\in{\cal S}_2^{nz}$ then $\cal N$ does not have \ac subnetworks of three reactions. Conversely, suppose  $\{a\to a', b\to b', c\to c', d\to d'\}$ is a subnetwork of $\cal N$ that is in \ac, that $a\to a', c\to c'$  form the $(\gets,\to)$ pattern on ${\cal N}_{\{i\}}$ with  $c_i<a_i$ and $b\to b', d\to d'$  form the $(\to,\gets)$ pattern on ${\cal N}_{\{j\}}$ with  $d_j<b_j$. If $i=j$ it is easy to check that either $\{a\to a', b\to b', c\to c', d\to d'\}\in$\at (a contradiction: \ac networks have \ac subnetworks composed of three reactions) or that ${\cal N}_{\{i\}}$ has the arrow diagram $(\leftarrow\!\!\bullet\!\!\rightarrow,\leftarrow\!\!\bullet\!\!\rightarrow)$. 
Suppose $i\neq j$, and suppose $v_i=v_j=1$ (the case $v_i=-v_j=1$ is similar). If $a_i>b_i$ then $\{a\to a', b\to b', d\to d'\}\in$\ac, contradiction. If $a_i<b_i$ then ${\cal N}_{\{i\}}\in$\at, contradiction. Therefore $a_i=b_i$. Likewise, $c_i=d_i$. 

Therefore all 1D projections of $\{a\to a', b\to b', c\to c', d\to d'\}$ have the arrow diagram $(\leftarrow\!\!\bullet\!\!\rightarrow,\leftarrow\!\!\bullet\!\!\rightarrow)$. 
$\cal N$ cannot more source complexes on any 1D projection, as that would mean ${\cal N}\in$\at. 
Suppose for two reactions $e\to e'$, $f\to f'$ we have $e_1=f_1$. We show that $e=f$. Suppose $e_k\neq f_k$ for some $k\in S$. If $e\to e'$ and $f\to f'$ have opposite  directions, then they form one of the $(\to,\gets)$ and $(\gets, \to)$ patterns in ${\cal N}_{\{k\}}$. We can pick a reaction $g\to g'$ with $g_1\neq f_1$ such that  $\{e\to e', f\to f', g\to g'\}\in$\ac, contradiction, so $e_k=f_k$. Since $e\to e'$ and $f\to f'$ were arbitrary, it follows that all source complexes that have the same projection on $\{1\}$ are equal. Therefore ${\cal N}$ contains exactly two source complexes $a=b$, $c=d$ and ${\cal N}=\{a\to a', b\to b', c\to c', d\to d'\}$. 
If for some $i,j\in S$ $v_iv_j(c_i-a_i)(c_j-a_j)<0$ then $\{a\to a', c\to c'\}$ is a zigzag in ${\cal N}_{\{i,j\}}$, contradiction. Therefore ${\cal N}\in{\cal S}_2^{nz}$.
\end{proof}

\smallskip
\begin{defn}{\em (Class $\cal C$: corner networks})\label{defn:class2}
A rank 1 reaction network $\cal N$ with is called {a $\cal C$ network} if it is essential, 1-alt complete, it contains at least three source complexes, and there exists $\gamma\in {\mathbb Z}^n$ such that 
$$a_k-\gamma\in\{0,e_1,\ldots,e_n\} \text{ for all } k\in R.$$ 
$\gamma$ is called the {\em corner of $\cal N$}. We write ${\cal N}\in {\cal C}.$
\end{defn}
Networks of class $\cal C$ are 1-alt complete with location of source complexes illustrated in Figure \ref{fig:C2D}{\it(iii)}; $\gamma$ may or may not be a source complex. Up to reversing arrows and reflection about the diagonal there are two possible arrow diagrams for 2D projections of $\cal C$ networks, shown in Figure \ref{fig:C2D}{\it (i),(ii)}. All 1-alt complete projections ($\ge2 D$) of $\cal C$ networks are $\cal C$ networks. However, in general not all $\ge 2D$ projections of $\cal N$ are in $\cal C$. 
One may check if a given network $\cal N$ is in class $\cal C$ by computing the vector 
\begin{equation}\label{eq:computeCorner}
\gamma=(\min_{k\in R}a_{k2}, \min_{k\in R}a_{km},\ldots,\min_{k\in R}a_{km}),
\end{equation}
for if ${\cal N}\in {\cal C}$ then necessarily $\gamma$ is the corner of $\cal N$.
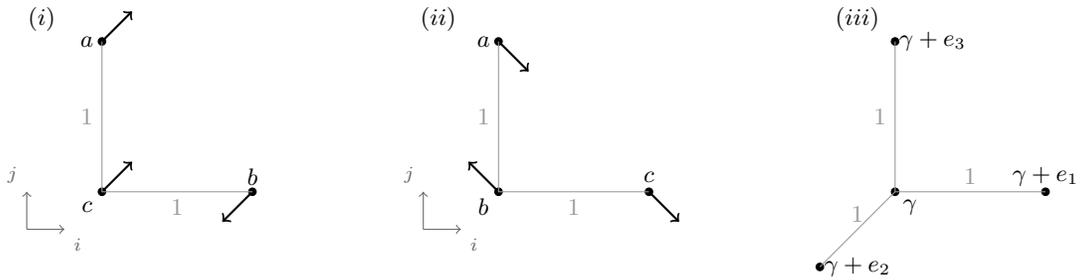
\begin{figure}[h]
\centering
\begin{tikzpicture}[scale=1]
\begin{scope}
\draw [->, color=black!60](-1,-.5)--(-.5,-.5) node[anchor=north west]{\scriptsize $i$};
\draw [->, color=black!60](-1,-.5)--(-1,0) node[anchor=south east]{\scriptsize $j$};
\draw[fill=black] (2,0) circle[ radius=0.05];
\draw[fill=black] (0,2) circle[ radius=0.05];
\draw[fill=black] (0,0) circle[ radius=0.05];
\draw [->,  thick](0,2)--(.4,2.4);
\draw [->,  thick](2,0)--(1.6,-.4) ;
\draw [->,  thick](0,0)--(.4,.4);
\draw [color=black!40] (2,0)--(0,0)--(0,2);
\draw[color=black!40] (1,-.2) node {\small $1$}; 
\draw[color=black!40] (-.2,1) node {\small $1$}; 
\draw (-.2,2) node {\small $a$};
\draw (2,.2) node {\small $b$};
\draw (-.2,-.2) node {\small $c$};
\draw (-.8,2.3) node {\small $(i)$};
\end{scope}
\begin{scope}[xshift=150]
\draw [->, color=black!60](-1,-.5)--(-.5,-.5) node[anchor=north west]{\scriptsize $i$};
\draw [->, color=black!60](-1,-.5)--(-1,0) node[anchor=south east]{\scriptsize $j$};
\draw[fill=black] (2,0) circle[ radius=0.05];
\draw[fill=black] (0,2) circle[ radius=0.05];
\draw[fill=black] (0,0) circle[ radius=0.05];
\draw [->,  thick](0,2)--(.4,1.6);
\draw [->,  thick](2,0)--(2.4,-.4) ;
\draw [->,  thick](0,0)--(-.4,.4);
\draw [color=black!40] (2,0)--(0,0)--(0,2);
\draw[color=black!40] (1,-.2) node {\small $1$}; 
\draw[color=black!40] (-.2,1) node {\small $1$}; 
\draw (-.2,2) node {\small $a$};
\draw (2,.2) node {\small $c$};
\draw (-.2,-.2) node {\small $b$};
\draw (-.8,2.3) node {\small $(ii)$};
\end{scope}
\begin{scope}[xshift=300]
\draw[fill=black] (2,0) circle[ radius=0.05];
\draw[fill=black] (0,2) circle[ radius=0.05];
\draw[fill=black] (0,0) circle[ radius=0.05];
\draw[fill=black] (-1,-1) circle[ radius=0.05];

\draw [color=black!40] (-1,-1)--(0,0)--(0,2);
\draw [color=black!40] (0,0)--(2,0);
\draw[color=black!40] (1,.2) node {\small $1$}; 
\draw[color=black!40] (-.2,1) node {\small $1$}; 
\draw[color=black!40] (-.5,-.3) node {\small $1$}; 
\draw (.5,2) node {\small $\gamma+e_3$};
\draw (2,.2) node {\small $\gamma+e_1$};
\draw (.2,-.2) node {\small $\gamma$};
\draw (-.5,-1) node {\small $\gamma+e_2$};
\draw (-.5,2.3) node {\small $(iii)$};
\end{scope}
\end{tikzpicture}

\caption{{\em Class $\cal C$.} {\it (i),(ii):} 1-alt complete 2D projections, drawn up to reflections about the diagonal and reversing all arrows. {\it(iii):} the possible locations of the source complexes when $n=3$ are $\gamma, \gamma+e_1, \gamma+e_2, \gamma+e_3$.}\label{fig:C2D}
\end{figure}

\begin{rmk}{\em (Relations between classes of networks).}\label{rem:diagram} It is clear that networks of classes $\cal L$, ${\cal S}_2^z$ and ${\cal S}_2^{nz}$ are 1-alt complete. Figure \ref{fig:A1networks} shows the set-theoretical relations between the various classes of networks defined in the paper.
\end{rmk}

\begin{figure}[h]
\tikzstyle{line}=[draw,-latex']
\centering
\begin{tikzpicture}[]
\node  [draw,minimum width= 16cm, minimum height=5cm, fill=gray!5!white] at (0,0) {};
\node  at (7.2,2) {\textbf{\ac}};
\node  [draw,minimum width=9cm, minimum height=4cm, fill=gray!10!white] at (-3,0) {};
\node  at (-7.2,1.7) {${\cal Z}$ };
\node  [draw,minimum width=2cm, minimum height=1.5cm, fill=gray!20!white]  at (-4,0) {};
\node  at (-4,0) {${\cal L}$};
\node  [draw,minimum width=2cm, minimum height=1.5cm, fill=gray!20!white]  at (-1,0) {};
\node  at (-1,0) {${\cal S}_2^z$ };
\node  [draw,minimum width=2cm, minimum height=1.5cm, fill=gray!20!white, fill opacity=.6]  at (2,0) {};
\node  at (2,0) {$\cal C$};
\node  [draw,minimum width= 2cm, minimum height=1.5cm, fill=gray!20!white, fill opacity=.6]  at (5,0) {};
\node  at (5,0) {${\cal S}_2^{nz}$}; 
\end{tikzpicture}
\caption{Inclusion relations between classes of networks \ac, ${\cal Z},\ {\cal S}_2^z, \cal L, {\cal S}_2^z \text{ and } \cal C$.}\label{fig:A1networks}
\end{figure}
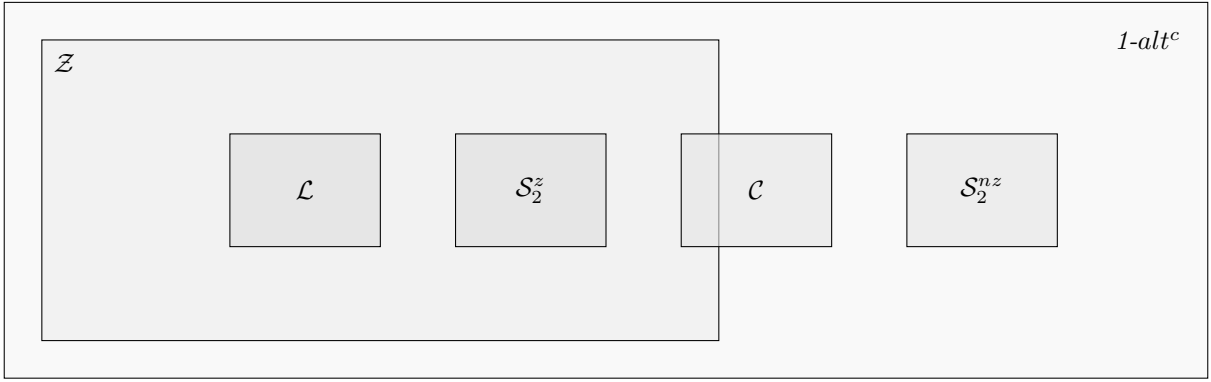


\section{Main result, examples, and discussion}\label{sec:main}

\subsection{Characterization of capacity for MPE/MPNE for rank 1 reaction networks}

\begin{thm}\label{thm:mainext}
Let ${\cal N}$ be a rank 1 reaction network. 
Then 
\begin{itemize} 
\item[(1)] If $\cal N$ has capacity for MPE then either $\cal N$ is 1-alt complete or ${\cal N}\in{\cal S}_1$.
\item[(2)] If ${\cal N}\in{\cal S}_1$ or if the essential projection of ${\cal N}$ is in ${\cal D}={\cal S}_2^z\cup{\cal L}\cup {\cal S}_2^{nz}\cup {\cal C}$ then $\cal N$ has the capacity for MPE, but it does not have the capacity for MPNE. 
\item[(3)] If the essential projection of $\cal N$ is 1-alt complete and not in $\cal D$ then $\cal N$ has the capacity for MPNE.
\end{itemize}
\end{thm}

\begin{ex}
Consider the network $\cal N$ on 6
species:
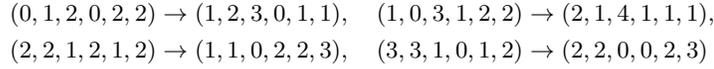
\begin{eqnarray*}
&&(0,1,2,0,2,2)\to(1,2,3,0,1,1), \quad (1,0,3,1,2,2)\to (2,1,4,1,1,1),\\
&&(2,2,1,2,1,2)\to(1,1,0,2,2,3), \quad (3,3,1,0,1,2)\to (2,2,0,0,2,3)
\end{eqnarray*}

The stoichiometric space of $\cal N$ is spanned by $v=(1,1,1,0,-1,-1)$. Since the sixth coordinate is the same for all source complexes, the essential projection of $\cal N$ is ${\cal N}_{\{1,2,3,5\}}$ and rename it as $\cal N$.
The four 1D projections of ${\cal N}$ are as follows:
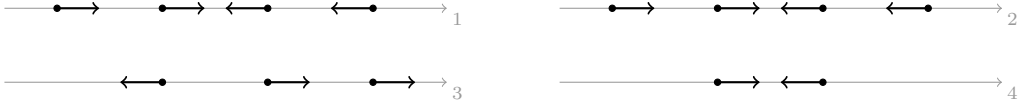
\begin{center}
\begin{tikzpicture}[scale=1.4][h]
\draw [->, black!40](-.5,0)--(3.7,0);
\draw (3.8,-.1) node[color=black!40] {\scriptsize $1$}; 
\draw[fill=black] (0,0) circle[ radius=0.03];
\draw[fill=black] (2,0) circle[ radius=0.03];
\draw[fill=black] (1,0) circle[ radius=0.03];
\draw[fill=black] (3,0) circle[ radius=0.03];
\draw [->, thick](0,0)--(.4,0);
\draw [->, thick](1,0)--(1.4,0);
\draw [->, thick](2,0)--(1.6,0);
\draw [->, thick](3,0)--(2.6,0);
\begin{scope}[yshift=-20]
\draw [->, black!40](-.5,0)--(3.7,0);
\draw (3.8,-.1) node[color=black!40] {\scriptsize $3$}; 
\draw[fill=black] (1,0) circle[ radius=0.03];
\draw[fill=black] (2,0) circle[ radius=0.03];
\draw[fill=black] (3,0) circle[ radius=0.03];
\draw [->, thick](1,0)--(.6,0);
\draw [->, thick](2,0)--(2.4,0);
\draw [->, thick](3,0)--(3.4,0);
\end{scope}
\begin{scope}[xshift=150]
\draw [->, black!40](-.5,0)--(3.7,0);
\draw (3.8,-.1) node[color=black!40] {\scriptsize $2$}; 
\draw[fill=black] (0,0) circle[ radius=0.03];
\draw[fill=black] (1,0) circle[ radius=0.03];
\draw[fill=black] (2,0) circle[ radius=0.03];
\draw[fill=black] (3,0) circle[ radius=0.03];
\draw [->, thick](0,0)--(.4,0);
\draw [->, thick](1,0)--(1.4,0);
\draw [->, thick](2,0)--(1.6,0);
\draw [->, thick](3,0)--(2.6,0);
\end{scope}
\begin{scope}[xshift=150,yshift=-20]
\draw [->, black!40](-.5,0)--(3.7,0);
\draw (3.8,-.1) node[color=black!40] {\scriptsize $4$}; 
\draw[fill=black] (1,0) circle[ radius=0.03];
\draw[fill=black] (2,0) circle[ radius=0.03];
\draw [->, thick](1,0)--(1.4,0);
\draw [->, thick](2,0)--(1.6,0);
\end{scope}
\end{tikzpicture}
\end{center}
We see that ${\cal N}\in$\ac, so $\cal N$ has the capacity for MPE. Since ${\cal N}_{\{1\}}$ contains four source complexes and $\cal N$ is not on two species
${\cal N}\notin {\cal D}$. $\cal N$ has the capacity for MPNE by theorem \ref{thm:mainext}.
\end{ex}

More examples of arrow diagrams and their capacity for MPE/MPNE are given in Figure \ref{fig:2Dexamples}. These include at least one network from each of the sets forming the partition of \ac in Figure \ref{fig:A1networks}.

\begin{figure}[h]
\centering
\begin{tikzpicture}[scale=.7]
\begin{scope}[scale=1.25]
\draw[fill=black] (2,1) circle[ radius=0.05];
\draw[fill=black] (0,2) circle[ radius=0.05];
\draw[fill=black] (2,2) circle[ radius=0.05];
\draw [->,  thick](2,1)--(2.4,1.4);
\draw [->,  thick](0,2)--(-.4,1.6) ;
\draw [<->, thick](1.6,1.6)--(2.4,2.4);
\draw [color=black!40] (2,1)--(2,2)--(0,2);
\draw (-.5,2.3) node {\scriptsize $(a)$};
\end{scope}
\begin{scope}[xshift=150, scale=1.25]
\draw[fill=black] (2,0) circle[ radius=0.05];
\draw[fill=black] (0,2) circle[ radius=0.05];
\draw[fill=black] (0,0) circle[ radius=0.05];
\draw [->,  thick](2,0)--(2.4,.4);
\draw [->,  thick](0,2)--(-.4,1.6) ;
\draw [<->,  thick](-.4,-.4)--(.4,.4);
\draw [color=black!40] (2,0)--(0,0)--(0,2);
\draw[color=black!40] (1,-.2) node {\scriptsize $1$}; 
\draw[color=black!40] (-.2,1) node {\scriptsize $1$}; 
\draw (-.5,2.3) node {\scriptsize $(b)$};
\end{scope}
\begin{scope}[xshift=300, scale=1.25]
\draw[fill=black] (2,0) circle[ radius=0.05];
\draw[fill=black] (0,2) circle[ radius=0.05];
\draw[fill=black] (0,0) circle[ radius=0.05];
\draw [->,  thick](2,0)--(2.4,.4);
\draw [->,  thick](0,2)--(-.4,1.6) ;
\draw [<->,  thick](-.4,-.4)--(.4,.4);
\draw [color=black!40] (2,0)--(0,0)--(0,2);
\draw[color=black!40] (1,-.2) node {\scriptsize $2$}; 
\draw[color=black!40] (-.2,1) node {\scriptsize $2$}; 
\draw (-.5,2.3) node {\scriptsize $(c)$};
\end{scope}
\begin{scope}[xshift=450, scale=1.25]
\draw[fill=black] (2,0) circle[radius=0.05];
\draw[fill=black] (0,2) circle[radius=0.05];
\draw[fill=black] (1,1) circle[radius=0.05];
\draw [->,  thick](2,0)--(2.4,.4);
\draw [->,  thick](0,2)--(-.4,1.6) ;
\draw [<->,  thick](.6,.6)--(1.4,1.4);
\draw [color=black!40] (2,0)--(0,0)--(0,2)--(2,2)--(2,0)--(0,2);
\draw (-.5,2.3) node {\scriptsize $(d)$};
\end{scope}
\begin{scope}[yshift=-130, scale=1.25]
\draw[fill=black] (2,0) circle[radius=0.05];
\draw[fill=black] (0,2) circle[radius=0.05];
\draw[fill=black] (1,1) circle[radius=0.05];
\draw [<->,  thick](1.6,-.4)--(2.4,.4);
\draw [<->,  thick](.4,2.4)--(-.4,1.6) ;
\draw [<->,  thick](.6,.6)--(1.4,1.4);
\draw [color=black!40] (2,0)--(0,0)--(0,2)--(2,2)--(2,0)--(0,2);
\draw (-.5,2.3) node {\scriptsize $(e)$};
\end{scope}
\begin{scope}[xshift=150, yshift=-130, scale=1.25]
\draw[fill=black] (2,0) circle[radius=0.05];
\draw[fill=black] (0,2) circle[radius=0.05];
\draw [<->,  thick](1.6,-.4)--(2.4,.4);
\draw [<->,  thick](.4,2.4)--(-.4,1.6) ;
\draw[color=black!40](2,0)--(0,0)--(0,2)--(2,2)--(2,0);
\draw (-.5,2.3) node {\scriptsize $(f)$};
\end{scope}
\begin{scope}[xshift= 300, yshift=-130, scale=1.25]
\draw[fill=black] (2,0) circle[radius=0.05];
\draw[fill=black] (0,1) circle[radius=0.05];
\draw [<->,  thick](1.6,-.4)--(2.4,.4);
\draw [<->,  thick](.4,1.4)--(-.4,.6) ;
\draw[color=black!40](2,0)--(0,0)--(0,1)--(2,1)--(2,0);
\draw (-.5,2.3) node {\scriptsize $(g)$};
\end{scope}
\begin{scope}[xshift=450, yshift=-130, scale=1.25]
\draw[fill=black] (2,0) circle[radius=0.05];
\draw[fill=black] (0,1) circle[radius=0.05];
\draw [<->,  thick](1.6,.4)--(2.4,-.4);
\draw [<->,  thick](-.4,1.4)--(.4,.6) ;
\draw[color=black!40](2,0)--(0,0)--(0,1)--(2,1)--(2,0);
\draw (-.5,2.3) node {\scriptsize $(h)$};
\end{scope}
\begin{scope}[yshift=-260, scale=1.25]
\draw[fill=black] (0,0) circle[radius=0.05];
\draw[fill=black] (2,1) circle[radius=0.05];
\draw [<->,  thick](1.6,1.4)--(2.4,.6);
\draw [<->,  thick](-.4,.4)--(.4,-.4) ;
\draw[color=black!40](2,0)--(0,0)--(0,1)--(2,1)--(2,0);
\draw (-.5,2.3) node {\scriptsize $(i)$};
\end{scope}
\begin{scope}[xshift=150, yshift=-260, scale=1.25]
\draw[fill=black] (2,0) circle[radius=0.05];
\draw[fill=black] (0,1) circle[radius=0.05];
\draw [<->,  thick](1.6,.4)--(2.4,-.4);
\draw [->,  thick](0,1)--(.4,.6) ;
\draw[color=black!40](2,0)--(0,0)--(0,1)--(2,1)--(2,0);
\draw (-.5,2.3) node {\scriptsize $(j)$};
\end{scope}
\begin{scope}[xshift=300, yshift=-260, scale=1.25]
\draw[fill=black] (2,0) circle[radius=0.05];
\draw[fill=black] (0,1) circle[radius=0.05];
\draw[fill=black] (0,0) circle[radius=0.05];
\draw [<-,  thick](2.4,-.4)--(2,0);
\draw [<-,  thick](.4,.6)--(0,1);
\draw [->,  thick](0,0)--(-.4,.4);
\draw[color=black!40](2,0)--(0,0)--(0,1)--(2,1)--(2,0);
\draw (-.5,2.3) node {\scriptsize $(k)$};
\end{scope}
\begin{scope}[yshift=-260,xshift=450, scale=1.25]
\draw[fill=black] (2,0) circle[radius=0.05];
\draw[fill=black] (0,2) circle[radius=0.05];
\draw[fill=black] (0,0) circle[radius=0.05];
\draw [<-,  thick](2.4,-.4)--(2,0);
\draw [<-,  thick](.4,1.6)--(0,2);
\draw [->,  thick](0,0)--(-.4,.4);
\draw[color=black!40](2,0)--(0,0)--(0,2)--(2,2)--(2,0);
\draw (-.5,2.3) node {\scriptsize $(l)$};
\draw[color=black!40] (1,-.2) node {\scriptsize $1$}; 
\draw[color=black!40] (-.2,1) node {\scriptsize $1$}; 
\end{scope}
\end{tikzpicture}
\caption{Examples of 2D arrow diagrams and their capacity for MPE/MPNE. 
$(a)$ ${\cal Z}\setminus({\cal S}_2^{nz}\cup{\cal L}\cup{\cal C})$: MPNE. 
$(b)$ ${\cal C}\cap{\cal Z}$: MPE, no MPNE. 
$(c)$ ${\cal Z}\setminus({\cal S}_2^{z}\cup{\cal L}\cup{\cal C})$: MPNE. 
$(d)$ ${\cal L}$: MPE, no MPNE. 
$(e)$: ${\cal Z}\setminus({\cal L}\cup{\cal C}\cup{\cal S}_2^{z})$: MPNE.
$(f)$ ${\cal S}_2^z$: MPE, no MPNE.
$(g)$ {\em (slope$\neq$-1)} ${\cal Z}\setminus({\cal S}_2^{z}\cup{\cal L}\cup{\cal C})$: MPNE. 
$(h)$ ${\cal S}_2^{nz}$: MPE, no MPNE.
$(i)$ ${\cal Z}\setminus({\cal S}_2^{z}\cup{\cal L}\cup{\cal C})$: MPNE.
$(j)$ not \ac: no MPE.
$(k)$ {\em (rectangle is not a square)} \ac$\setminus ({\cal Z}\cup{\cal C}\cup{\cal S}_2^{nz})$: MPNE.
$(l)$ ${\cal C}\setminus{\cal Z}$: MPE, no MPNE.
}
\label{fig:2Dexamples}
\end{figure}
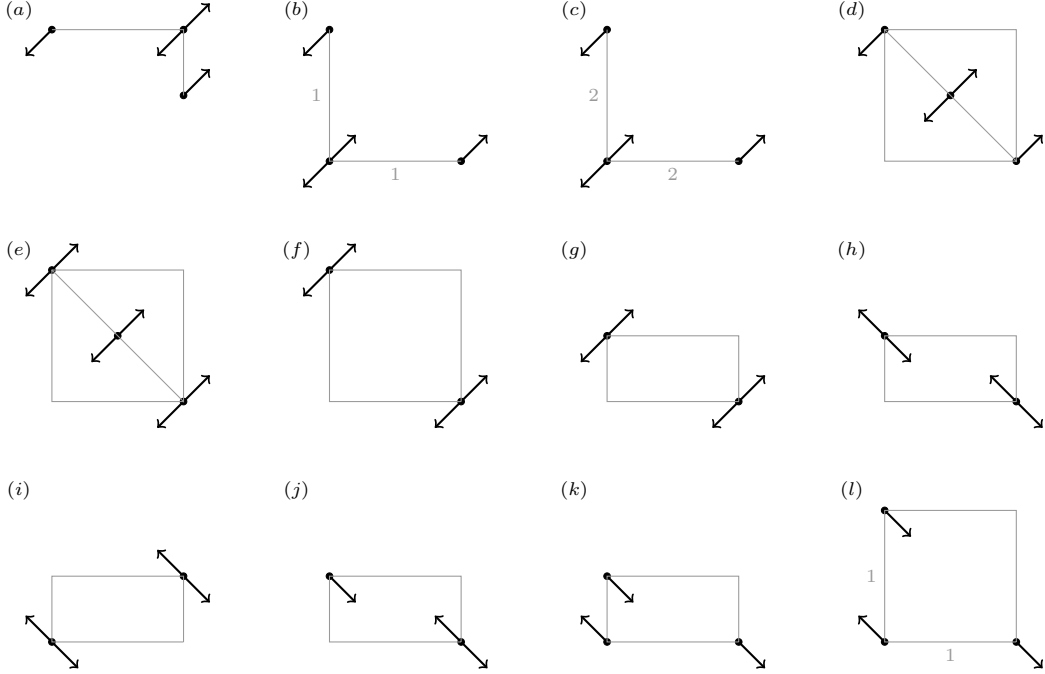

Theorem \ref{thm:mainext} has the following useful consequence.

\begin{cor}\label{cor:2alt} 
If ${\cal N}\in$\at then $\cal N$ has the capacity for MPNE. 
\end{cor}
\begin{proof}
If ${\cal N}\in \cal D$ then each 1D projection of $\cal N$ has two source complexes, contradicting ${\cal N}\in$\at. Therefore $\cal N$ has the capacity for MPNE by Theorem \ref{thm:mainext}.
\end{proof}
 
\subsection{Discussion and past results}\label{sec:disc}
Multistationarity of rank 1 reaction networks has been the focus of a number of recent papers \cite{JoshiShiu17, ShiuWolff19, TangXu21, Lin2022, TangZhang22}. These works study structural conditions for multistationarity (and/or multistability), and are connected to Conjecture \ref{conj:1}, which we prove in this paper. Joshi and Shiu showed that the conjecture holds for networks on one species, and for networks composed of two (possibly reversible) reactions. They also characterized the capacity for MPNE for networks on one species or networks with exactly two reactions. Theorem \ref{thm:mainext} extends this characterization of MPNE to the general case, revealing new structures with capacity for MPE, but not MPNE. We review the relevant results from the work of Joshi and Shiu \cite{JoshiShiu17} further down this section. 

Other partial results towards Conjecture \ref{conj:1} are known. Lin et al. \cite[Theorem 4.3]{Lin2022} proved that being 1-alt complete is a necessary condition for multistationarity under the assumption that the network cannot have an infinite number of positive equilibria for any choice of rate constants. This excludes all networks that have the capacity for MPE but not MPNE, but also some networks that have the capacity for MPNE (see Figure \ref{fig:2Dexamples}(e)). Lin et al. also showed that being 1-alt complete is sufficient for multistationarity if there exist two reactions in the network such that the subnetwork consisting of the two reactions admits at least one and finitely many positive steady states \cite[Theorem 4.7]{Lin2022}. Note that this technical condition cannot be immediately related with the structure of the network. 


Other work on multistationarity addressed conditions for the existence of three or more positive equilibria \cite{Lin2022}, conditions for multistability and the number of stable equilibria \cite{TangXu21, TangZhang22}, and conditions for the existence of nondegenerate multiple equilibria. Joshi and Shiu showed that (1) networks with one irreversible and one reversible reaction and  (2) networks with two reversible reactions must necessarily have rank 1 to be multistationary. They also characterized the capacity for MPE for these networks \cite[Theorems 2.8, 5.12]{JoshiShiu17}. In subsequent work Shiu and de Wolff characterized the capacity for MPNE for networks of type (1) and (2) above for the special case of two species \cite{ShiuWolff19}. These results have quick proofs based on Theorem \ref{thm:mainext}, and so does the general case (any number of species), posed as an open question in \cite{ShiuWolff19}:

\begin{prop}[\cite{ShiuWolff19} Conjecture 5.1]
A network $\cal N$ consisting of one pair of reversible reactions and one irreversible reaction has the capacity for MPNE if and only if ${\cal N}\in$\at.
\end{prop}
\begin{proof} Assume without loss of generality that ${\cal N}=\{a\rightleftharpoons b, c\to c'\}$ is an essential network.  Suppose ${\cal N}$ has the capacity for MPNE. Then $\cal N$ must have rank 1 \cite{JoshiShiu17}, discussed above. By Theorem \ref{thm:mainext}, $\cal N$ must be 1-alt complete, which implies that $c\notin\{a,b\}$.
It is clear that ${\cal N}\notin {\cal L}\cup{\cal S}_2^z$ (for these classes reaction vectors cannot be parallel to the line of source complexes) and that ${\cal N}\notin {\cal S}_2^{nz}$ (that requires four reactions). Moreover, ${\cal N}\notin\cal C$, since then it would have arrow diagram in Figure \ref{fig:C2D}(i),(ii), contradiction. The other implication follows from Corollary \ref{cor:2alt}.
\end{proof}

The same argument can be applied to characterize the capacity for MPNE for networks composed of two reversible reactions and any number of species, generalizing Theorem 3.6 in \cite{ShiuWolff19}:

\begin{prop}
A network $\cal N$ consisting two pairs of reversible reactions has the capacity for MPNE if and only if $\cal N$ has rank 1 and ${\cal N}\in$\at  (equivalently, a 1D projection of $\cal N$ contains the $(\to,\gets,\to,\gets)$ pattern.
\end{prop}

\medskip 

\noindent {\bf A review of results for one and two species.} The next theorem rephrases some of the implications of Theorems 3,6, 5.1 and 5.2 in \cite{JoshiShiu17}. We state these using our terminology.

\begin{thm}\cite[Theorem 3.6]{JoshiShiu17}\label{thm:1species} Let $\cal N$ be a reaction network on one species. Then 
\begin{itemize}
\item[1.] $\cal N$ has the capacity for MPE if and only it is a 2-alternating network or the arrow diagram of $\cal N$ is either $\leftarrow\!\!\bullet\!\!\rightarrow$ or
$(\leftarrow\!\!\bullet\!\!\rightarrow,  \leftarrow\!\!\bullet\!\!\rightarrow)$
\item[2.] $\cal N$ has the capacity for MPNE if and only if $\cal N$ is a 2-alternating network.
\end{itemize}
\end{thm}

Lemma \ref{cor:31} shows that Theorem 5.2 in \cite{JoshiShiu17} can be rephrased as 

\begin{thm}\cite[Theorems 5.1, 5.2]{JoshiShiu17}\label{thm:Joshi,Shiu}
Let $\mathcal{N}=\{a\to a^{\prime}, b\to b'\}$ be 1-alt complete.  
\begin{enumerate}
    \item[1.] If ${\cal N}\notin {\cal S}_2^z$ then $\cal N$ has the capacity for  MNPE;
    \item[2.] If ${\cal N}\in {\cal S}_2^z$ then $\cal N$ has the capacity for MPE, but does not have the capacity for MPNE.
\end{enumerate}
\end{thm}

We conclude with a few consequences of Theorems \ref{thm:1species} and \ref{thm:Joshi,Shiu}  used later in our proof.

\begin{cor}\label{cor:2alt}
If $\cal N$ is 2-alternating then $\cal N$ has the capacity for MPNE.
\end{cor}

\begin{cor}\label{cor:no2zz2}
Let $\mathcal{N}=\{a\to a^{\prime}, b\to b'\}$ be 1-alt complete and suppose $\cal N$ has zigzags in two different 2D projections. Then $\cal N$ has the capacity for MPNE.
\end{cor}



\section{Proofs}\label{sec:proofs}  The proof of parts (1) and (2) of Theorem \ref{thm:mainext} are self-contained and given in sections  \ref{sec:necessary} and \ref{sec:degen}. The proof of part (3) uses results of Joshi and Shiu for one species and two reactions, and it also uses in a critical way a new inheritance theorem due to Banaji et.~al. (Theorem \ref{thm:inheritance2} in this paper). Our original proof of  part 3 used a hands-on analysis of roots of $f(\eta,\kappa;t)$ and provided us a more algebraic insight \cite{Pantea.2022hj}. On the other hand, that proof is longer and less streamlined than the one we give here. We start with a proposition that allows us to assume $v$ has only 1 and -1 entries, and that $\lambda_k \in\{1,-1\}$, which simplifies some our our calculations. 
\begin{prop}\label{reduc to 1} Let $\cal N$ be an essential rank 1 reaction network. Denote $|v|=(|v_1|,\ldots, |v_n|)$ and let $\tilde v=sign(v)$ and $\tilde{\eta}=\frac{\eta}{|v|}$ (componentwise division). 
Let $\tilde \lambda_k=\text{sign}\lambda_k$, $\tilde \lambda_k^+=sign(\lambda_k^+)$, $\tilde \lambda_k^-=sign(\lambda_k^-)$.
If $k\in R\setminus R^0$ we let $\tilde{\kappa}_k=\kappa_k|\lambda_k||v|^{a_k}$, and if $k\in R^0$ we let $\tilde{\kappa}_k^+=\kappa_k^+|\lambda_k^+||v|^{a_k}$ and $\tilde{\kappa}_k^-=\kappa_k^-|\lambda_k^-||v|^{a_k}$. Let
$\tilde{\cal N}$ denote the reaction network with reactions $a_k\to a_k+\tilde\lambda_k\tilde v$ for $k\in R\setminus R^0$ and $a_k\to a_k+\tilde\lambda_k^+\tilde v$, $a_k\to a_k+\tilde\lambda_k^-\tilde v$ for $k\in R^0$. 
Then $\eta+t_0 v$ is a (nondegenerate/degenerate) positive equilibrium of $(\cal N, \kappa)$ if and only if
$\tilde \eta+t_0\tilde v$ is a (nondegenerate/degenerate) positive equilibrium of $(\tilde{\cal N},\tilde\kappa)$.
\end{prop}

\begin{proof} We have
\begin{eqnarray*}
f_{\cal N}(\eta,\kappa;t)&=&
\sum_{k\in R\setminus R^0}\lambda_k\kappa_k(\eta+tv)^{a_k}
+\sum_{k\in R^0}(\lambda_k^+\kappa_k^+-\lambda_k^-\kappa_k^-) (\eta+tv)^{a_k}\\
 &=&\sum_{k\in R\setminus R^0}\kappa_k\lambda_k\prod_{i\in S}|v_i|^{a_{ki}}\prod_{i\in S}(\tilde{\eta_i}+t\tilde{v}_i)^{a_{ki}}
 +\sum_{k\in R^0}(\lambda_k^+\kappa_k^+-\lambda_k^-\kappa_k^-) \prod_{i\in S}|v_i|^{a_{ki}}\prod_{i\in S}(\tilde{\eta_i}+t\tilde{v}_i)^{a_{ki}}
 \\
 &=&\sum_{k\in R\setminus R^0}\tilde{\kappa}_k\tilde\lambda_k(\tilde\eta+t\tilde v)^{a_k}+
 \sum_{k\in R^0}(\tilde\lambda_k^+\tilde\kappa_k^+-\tilde\lambda_k^-\tilde\kappa_k^-)(\tilde\eta+t\tilde v)^{a_k}\\
 &=&f_{\tilde {\cal N}}(\tilde\eta,\tilde\kappa;t)
\end{eqnarray*}
and the conclusion follows from Proposition \ref{prop:DSSf}.
\end{proof}



\subsection{Proof of Theorem \ref{thm:mainext} (1): necessary condition for multistationarity}\label{sec:necessary} 
Assume that $\cal N$ is not a single-source network and that it does not have 1D projections containing the arrow pattern $(\rightarrow, \leftarrow)$ (the case when $\cal N$ does not have 1D projections containing the arrow pattern $(\leftarrow, \rightarrow)$ is proved similarly.) We write
\begin{equation}\label{eq: vecor field not A_1}
    f(\eta,\kappa;t)=\sum_{k\in R^+}\kappa_k(\eta+tv)^{a_k}-\sum_{l\in R^-}\kappa_l(\eta+tv)^{a_l}.
\end{equation}
If $R^+=\emptyset$ or $R^-=\emptyset$ then $f(\eta,\kappa;t)$ is either strictly positive or strictly negative on $(\alpha^\eta,\beta^\eta)$ and therefore $f(\eta,\kappa;t)=0$ does not have solutions.

If $R^+\neq\emptyset$ and $R^-\neq\emptyset$, then (see (\ref{eq:><explained})) for any $k\in R^+, l\in R^-$ and $i\in S$  we have 
$(a_{ki}-a_{li})v_i\geq 0.$ In other words for $i\in S^+$ we have 
$\displaystyle\max_{l\in R^-} a_{li}\le \displaystyle\min_{k\in R^+} a_{ki}$ 
and for $i\in S^-$ we have 
$\displaystyle\max_{k\in R^+} a_{ki}\le \displaystyle\min_{l\in R^-} a_{li}.$

Let $\delta\in\mathbb{R}^n$ be such that  that 
\begin{eqnarray}\label{eq:alphachoice}
\underset{l\in R^-}{\max}\ a_{li}\leq\delta_i\leq\underset{k\in R^+}\min\ a_{ki}\; \text{ if } i\in S^+\\\nonumber
\underset{k\in R^+}{\max}\ a_{ki}\leq\delta_i\leq \underset{l\in R^-}{\min}\ a_{li}\; \text { if } i\in S^-.
\end{eqnarray}

Since $\cal N$ has at least two different source complexes there must exist $(k_0,l_0)\in R^+\times R^-$ such that $a_{k_0}\neq a_{l_0}$. Let $i_0\in S$ be such that  $a_{k_0i_0}\neq a_{l_0i_0}$. 
We  refine our choice of $\delta$ (\ref{eq:alphachoice}) to ensure that 
$\delta_{i_0}-a_{k_0i_0}\neq 0$ and $\delta_{i_0}-a_{l_0i_0}\neq 0$.
Setting  (\ref{eq: vecor field not A_1}) equal to zero yields
\begin{equation}\label{eq:1}
\sum_{k\in R^+}\kappa_k(\eta+tv)^{a_k-\delta}=\sum_{l\in R^-}\kappa_l(\eta+tv)^{a_l-\delta}.
\end{equation}
Note that exponents $a_k-\delta$ on the left hand side of (\ref{eq:1}) have coordinates of the same sign or zero, with at least one (i.e. $a_{k_0i}-\delta_i$) being nonzero. The same holds for the right hand side of (\ref{eq:1}), and the sign there is the opposite of the sign in the left hand side. It follows that the two sums in equality $(\ref{eq:1})$ are strictly monotonic functions of $t$ on $(\alpha^\eta, \beta^\eta)$ with different monotonicities. Therefore $(\ref{eq:1})$ has at most one solution in $t$, and ${\cal N}$ does not have the capacity for multiple positive equilibria.     
\subsection{Proof of Theorem \ref{thm:mainext} (2): capacity for MPE, but not MPNE} \label{sec:degen}
\noindent {\bf Case ${\cal N}\in {\cal S}_1$.} Denote the rate constants by $\kappa^+$ and $\kappa^-$. positive Equilibria exist if and only if $\kappa^+=\kappa^-$. In that case $f(\eta,\kappa;t)$ is identically zero for any $\eta$ and all equilibria are degenerate.

\medskip

\noindent {\bf Case ${\cal N}\in {\cal S}_2^{z}$.} $\cal N$ is a zigzag subnetwork of that in Figure \ref{fig:s2z}(i). Consider the maximal configuration first, i.e. let 
${\cal N}=\{a\to a', a\to a'', b\to b', b\to b''\}$, and denote the corresponding rate constants by $\kappa_a^+, \kappa_a^-,\kappa_b^+,\kappa_b^-$ respectively. Let $d=a_2-b_2=b_1-a_1\neq 0$. The equation
$f(\eta,\kappa;t)=(\kappa_a^+-\kappa_a^-)(\eta+tv)^a+(\kappa_b^--\kappa_b^-)(\eta+tv)^b=0$ is equivalent to
\begin{eqnarray}\label{eq:55}\nonumber
h(\eta,\kappa;t)&=&  (\kappa_a^+-\kappa_a^-)+(\kappa_b^--\kappa_b^-)(\eta+tv)^{b-a}\\
&=&(\kappa_a^+-\kappa_a^-)+(\kappa_b^--\kappa_b^-)
\left(\frac{\eta_1+t}{\eta_2+t}\right)^d=0.
\end{eqnarray}
If $\kappa_b^+=\kappa_b^-$ then $h(\eta,\kappa;t)$ has zeros if and only if it is identically zero, in which case $f(\eta,\kappa;t)$ and $\cal N$ has a continuum of positive equilibria, all degenerate. 

If $\kappa_b^+\neq \kappa_b^-$ then if $\eta_1\neq \eta_2$   $h(\eta,\kappa;t)$ is strictly monotonic and has at most one zero. On the other hand, if $\eta_1=\eta_2$ then equation (\ref{eq:55}) has solutions if and only if $(\kappa_a^+-\kappa_a^-)+(\kappa_b^--\kappa_b^-)=0$, in which case
$f(\eta,\kappa;t)$ is identically zero, and $\cal N$ has multiple equilibria, all degenerate. 

This completes the proof for the case of arrow diagram in Figure \ref{fig:s2z}(i). If $\cal N$ is a strict subnetwork of $\{a\to a', a\to a'', b\to b', b\to b''\}$ then calculation above applies as well, by setting the rate constants of missing reactions to zero, and the conclusion stays the same: $\cal N$ has multiple equilibria only on compatibility classes with $\eta_1=\eta_2$, and only if the sum of rate constants for all reactions in one direction is equal to the the sum of rate constants on the other direction (this is possible since we have at least one reaction per direction). All equilibria are degenerate.  

\medskip

\noindent {\bf Case ${\cal N}\in {\cal L}$.}
Let $\delta\in{\mathbb R}^2$ such that 
$$\begin{pmatrix}
       a_{k1}\\a_{k2}
\end{pmatrix}=
\delta+p_k
\begin{pmatrix}
      1\\-1
\end{pmatrix}
\text{ for all } k\in R 
$$
where $p_k<0$ for $R\in R^+\setminus R_0$, $p_k=0$ for $k\in R^0$,  and $p_k>0$ for $k\in R^-\setminus R_0$.

Then 
\begin{equation}\label{eq:classIIf}
 h(\eta,\kappa;t):=\frac{f(\eta,\kappa;t)}{(\eta+tv)^{\delta}}=\sum_{k\in R ^+\setminus R^0}\kappa_k\left(\frac{\eta_1+t}{\eta_2+t}\right)^{p_k}
 +\sum_{k\in R^0}(\kappa_k^+-\kappa_k^-)
 -\sum_{l\in R^-\setminus R^0}\kappa_l\left(\frac{\eta_1+t}{\eta_2+t}\right)^{p_l}
\end{equation}

Since $R^+\setminus R_0$ and $R^-\setminus R_0$ are not both empty  $h(\eta,\kappa;t)$ is a strictly monotonic function if $\eta_1\neq \eta_2$, and $f(\eta,\kappa;t)$ has at most one zero. If $\eta_1=\eta_2$ then  
$h(\kappa ,\eta;t)=\sum_{k\in R^+\setminus R^0}\kappa_k+\sum_{k\in R^0}(\kappa_k^+-\kappa_k^-)-\sum_{l\in R^-\setminus R^0}\kappa_l$. Equilibria exist if and only if  and $h(\eta,\kappa;t)$ is identically zero, All equilibria are degenerate in this case.

\medskip
\noindent {\bf Case ${\cal N}\in {\cal S}_2^{nz}$.} Let ${\cal N}=\{a\to a', a\to a'', b\to b', b\to b''\}$ and denote the corresponding rate constants by $\kappa_a^+, \kappa_a^-,\kappa_b^+,\kappa_b^-$ respectively. 
Suppose $b-a\in{\cal Q}(v)$ (the other case can be reduced to this by changing $v$ to $-v$). Then  

$f(\eta,\kappa;t)=(\kappa_a^+-\kappa_a^-)(\eta+tv)^a+(\kappa_b^--\kappa_b^-)(\eta+tv)^b=0$ is equivalent to
\begin{eqnarray*}
h(\eta,\kappa;t)&=&  (\kappa_a^+-\kappa_a^-)+(\kappa_b^--\kappa_b^-)(\eta+tv)^{b-a}\\
&=&(\kappa_a^+-\kappa_a^-)+(\kappa_b^--\kappa_b^-)\prod_{k\in K^+}(\eta_k+t)^{b_k-a_k}
\prod_{l\in K^-}(\eta_l-t)^{b_l-a_l}=0.
\end{eqnarray*}

Since $b_k-a_k>0$ for all $k\in K^+$ and  $b_l-a_l<0$ for all $l\in K^-$, $h(\eta,\kappa;t)$ is a monotonically increasing function which has zeros if and only if it is identically zero, i.e. when $\kappa_a^+=\kappa_a^-$ and $\kappa_b^+=\kappa_b^-$. In this case $f(\eta,\kappa;t)$ is identically zero and all equilibria of $\cal N$ are degenerate. 

\medskip

\noindent {\bf Case ${\cal N}\in{\cal C}$.} By permuting species we assume $a_k=\gamma+e_k$ for $k\in R=\{1,\ldots, m\}$. Note that $R$ may not index all complexes (if $\gamma$ is a complex as well), but we can still use notation from section \ref{sec:notation}. We assume $\lambda_k, \lambda_k^+, \lambda_k^-\in\{-1,1\}$ (see Proposition \ref{reduc to 1}).

We write 
$$f(\eta,\kappa;t)=\epsilon d_0(\eta+tv)^{\gamma}
+\sum_{k\in R\setminus R^0}\lambda_k\kappa_k(\eta+tv)^{a_k}
+\sum_{k\in R^0}(\kappa_k^+-\lambda_k^-)\lambda_k^+(\eta+tv)^{a_k}.
$$
where  $\epsilon=1$ if $\gamma$ is a source complex of $\cal N$ and $\epsilon=0$ otherwise, and $d_0=\kappa_0\lambda_0$ if $\gamma$ is involved in a single reaction, and  $d_0=\kappa^+_0\lambda^+_0-\kappa^-_0\lambda^-_0$ if two reactions of opposite directions start at $\gamma$. Note that $d_0$ is imposed a sign if there is only one reaction with cource complex at $\gamma$.
Steady states of $\cal N$ correspond to zeros of the function
\begin{equation}
h(\eta,\kappa;t)=\dfrac{f(\kappa,\eta;t)}{(\eta+tv)^\gamma}=
\epsilon d_0
+\sum_{k\in R\setminus R^0}\kappa_k\lambda_k(\eta_k+tv_k)
+\sum_{k\in R^0}(\kappa_k^+-\kappa_k^-)\lambda_k^+(\eta_k+tv_k)=A_0+A_1t
\end{equation}
a polynomial in $t$ of degree at most one. Therefore $\cal N$ has multiple steady states if and only if $h$ is identically zero, in which case all steady states are degenerate. It remains to check that there are choices of $\eta\in{\mathbb R}^n_{>0}$ and of rate constants $\kappa$ such that $h$ is identically zero. 

The following trivial lemma is useful:
\begin{lem}\label{lem:+-}
Suppose $u\in {\mathbb R}^n$ has at least a positive coordinate and at least a negative coordinate. Then 
$$\{u\cdot w|w\in{\mathbb R}^n_{>0}\}=\mathbb R.$$
\end{lem}

We have 
$$A_0=\epsilon d_0+\sum_{k\in R\setminus R^0}\kappa_k\lambda_k\eta_{k}+
\sum_{k\in R^0}(\kappa_k^+-\kappa_k^-)\lambda_k^+\eta_{k}\text {  and  }A_1=\sum_{k\in R\setminus R^0}\kappa_k\lambda_kv_k
+\sum_{k\in R^0}(\kappa_k^+-\kappa_k^-)\lambda_k^+v_k.
$$

We treat the case $R=R^0$ separately: then $R^0$ has at least three elements. Let  $k_1\neq k_2\in R$.
If ${\cal N}$ does not not have a corner source complex (i.e. $\epsilon=1$) then set $\kappa_{k}^+=\kappa_{k}^-$ for all $k\in R$. Suppose $\epsilon =1$. 
Let $\kappa_k^+$, $\kappa_k^-$, $k\in R$ such that $A_1=0$ and such that $\kappa_{k_1}^+\neq \kappa_{k_1}^-$ and $\kappa_{k_2}^+\neq \kappa_{k_2}^-$ (we are not yet assigning these as rate constants). Then we choose any positive vector $\eta$ such that $\sigma=\sum_{k\in R^0}(\kappa_k^+-\kappa_k^-)\lambda_k^+\eta_{k}\neq 0$. If $d_0$ has fixed sign then pick any rate constant $k_0$. If $\sigma$ has opposite sign from that of $d_0$, choose the rate constants for $a_k\to a'^+_k$ and $a_k\to a'^-_k$ to be $\kappa_k^+/|d_0|$ and $\kappa_k^-/|d_0|$. If $\sigma$ has the same sign as $|d_0|$ then choose the rate constants for $a_k\to a'^+_k$ and $a_k\to a'^-_k$ to be $\kappa_k^-/|d_0|$ and $\kappa_k^+/|d_0|$.  

Now suppose $R\setminus R^0$ is nonempty. 

1. First suppose there exist $k,l\in R\setminus R^0$ such that $\lambda_kv_k$ and $\lambda _lv_l$ have different signs. If there are reactions with source complex $\gamma$, set their rate constants to $1$. This defines $d_0$. 
We fix rate constants for source complexes in $R^0$ such that the set
\begin{equation}\label{eq:set}
\{\kappa_k\lambda_k|k\in R\setminus R^0\}\cup\{(\kappa_{k}^+-\kappa_{k}^-)\lambda_{k}^+|k\in R^0\}
\end{equation}
contains elements of both signs. Note that this is automatically satisfied if $R^0=\emptyset$. We apply Lemma \ref{lem:+-} to find $\kappa_k>0$ for each $k\in R\setminus R^0$ such that $A_1=0$, and we apply Lemma \ref{lem:+-} again to find $\eta$ such that $A_0=0$. 

2. Now assume that $\lambda_kv_k$ have the same sign for all $k\in R\setminus R^0$, and assume this is positive (by changing $v$ to $-v$ if necessary). In other words, all reactions with source complexes in $R$ do in the positive direction in all 1D projections. We must have $R^0\neq\emptyset$ since ${\cal N}\in$\ac.
If $\{\lambda_k|k\in R\setminus R^0\}$ do not have the same sign then pick any rate constants for all reactions so that $A_1=0$ and apply Lemma \ref{lem:+-} to find $\eta$ such that $A_0=0$. 

Suppose $\{\lambda_k|k\in R\setminus R^0\}$ have the same sign, and suppose they are positive (by changing the direction of all reactions if necessary).
Then $v_k=1$ for all $k\in R\setminus R^0$. If $R^0$ has at least two elements, then let $k_0\in R^0$ and choose rate constants such that  $(\kappa_{k_0}^+-\kappa_{k_0}^-)\lambda_{k_0}^+<0$. Since there is an additional element of $R^0$ we may choose the remaining rate constants so that $A_1=0$; then we apply Lemma \ref{lem:+-} to find $\eta$ such that $A_0=0$. Suppose $|R_0|=\{k_0\}$ has a single element. If $v_{k_0}=1$ then we can choose rate constants such that $A_1=0$. In particular $(\kappa_{k_0}^+-\kappa_{k_0}^-)\lambda_{k_0}^+<0$ and the set (\ref{eq:set}) has elements of all signs. We apply Lemma \ref{lem:+-} to find $\eta$ such that $A_0=0$. Finally, if $v_{k_0}=-1$ then the 1D projections of the subnetwork $\tilde{\cal N}$ of $\cal N$ that leaves out all reactions from $\gamma$ have the following arrow diagrams: 

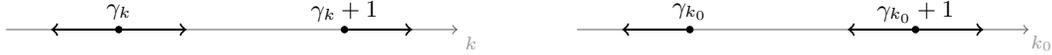
\begin{center}
\begin{tikzpicture}[scale=0.6][h]
\begin{scope}
\draw [->, black!60](0,0)--(10,0);
\draw[fill=black] (2.5,0) circle[radius=0.07];
\draw[fill=black] (7.5,0) circle[radius=0.07];
\draw [<->, thick](1,0)--(4,0);
\draw [->, thick](7.5,0)--(9,0);
\draw (2.5,.4) node {\small $\gamma_{k}$}; 
\draw (7.5,.4) node {\small $\gamma_{k}+1$}; 
\draw (10.3,-.3) node[color=black!40] {\scriptsize $k$}; 
\end{scope}
\begin{scope}[xshift=360]
\draw [->, black!60](0,0)--(10,0);
\draw[fill=black] (2.5,0) circle[radius=0.07];
\draw[fill=black] (7.5,0) circle[radius=0.07];
\draw [->, thick](2.5,0)--(1,0);
\draw [<->, thick](9,0)--(6,0);
\draw (2.5,.4) node {\small $\gamma_{k_0}$}; 
\draw (7.5,.4) node {\small $\gamma_{k_0}+1$}; 
\draw (10.3,-.3) node[color=black!40] {\scriptsize $k_0$}; 
\end{scope}
\end{tikzpicture}
\end{center}

$\tilde {\cal N}$ does not have the $(\to,\gets)$ pattern, and therefore there exists a reaction $\gamma\to\gamma+\lambda_0 v$ in $\cal N$ that produces the $(\to,\gets)$ pattern in ${\cal N}_{\{k_0\}}$, i.e. $\lambda_0>0$. Choose any rate constants such that $A_1=0$, and pick any $\eta\in{\mathbb R}^n_{>0}$. The elements of  $(\ref{eq:set})$ are all positive. Letting 
$\kappa_0=\sum_{k\in R\setminus R^0}\kappa_k\lambda_k\eta_{k}+
\sum_{k\in R^0}(\kappa_k^+-\kappa_k^-)\lambda_k^+\eta_{k}$
we have $d_0=-\kappa_0$ and so $A_0=0$. If there are two reactions from $\gamma$ we choose their rate constants such that $\kappa_0^--\kappa_0^+=\kappa_0$. 

\subsection{Proof of Theorem \ref{thm:mainext} (3): networks with the capacity for MPNE}


\subsubsection{Networks on two species} First we prove the statement for network on two species, $n=2$. We begin with a few lemmas; the first is an elementary exercise whose proof is omitted. For simplicity we let ${\cal D}={\cal S}_2^z\cup{\cal L}\cup{\cal C}\cup{\cal S}_2^{nz}$. 

\begin{lem}\label{lem:ginj}
Let $g:{\mathbb R}\to {\mathbb R}$ be a smooth function, and suppose $g(t_0)>0$, $g'(t_0)=0$ and $g''(t_0)\neq 0$. Then for any $\delta>0$ there exist $t_1,t_2\in (t_0-\delta, t_0+\delta)$, $t_1\neq t_2$ such that $g(t_1)=g(t_2)>0$ and $g'(t_1)\neq 0$, $g'(t_2)\neq 0$.
\end{lem}
\begin{lem}\label{lem:39}
Let $\cal N$ be a rank 1 reaction network on two species. 
\begin{itemize}
    \item[1.] If $\cal N$ has one of the arrow diagram in Figure \ref{fig:MPNEmotifs} (i), (iii), (iv), (v) then $\cal N$ has the capacity for MPNE. 
    \item[2.] If $\cal N$ has the arrow diagram in Figure \ref{fig:MPNEmotifs} (ii) or (v) and ${\cal N}\notin{\cal C}$ then $\cal N$ has capacity for MPNE. 
\end{itemize}

\end{lem}
\begin{figure}[h]
\centering
\begin{tikzpicture}[scale=.7]
\begin{scope}
\draw[fill=black] (2,0) circle[ radius=0.05];
\draw[fill=black] (0,2) circle[ radius=0.05];
\draw[fill=black] (2,3) circle[ radius=0.05];
\draw [->,  thick](0,2)--(.4,2.4);
\draw [->,  thick](2,0)--(1.6,-.4) ;
\draw [->,  thick](2,3)--(2.4,3.4);
\draw[color=black!40] (2,0)--(0,0);
\draw[color=black!40] (0,2)--(0,0);
\draw[color=black!40] (2,0)--(2,3.5);
\draw[thick,color=black!40] (0.96,-.13)--(.96,.13);
\draw[thick,color=black!40] (1.04,-.13)--(1.04,.13);
\draw[thick,color=black!40] (-.13,0.96)--(.13,.96);
\draw[thick,color=black!40] (-.13,1.04)--(.13,1.04);
\draw (-.2,2) node {\small $a$};
\draw (2.2,-.2) node {\small $b$};
\draw (2.2,3) node {\small $c$};
\draw (-1,2.5) node {\small $(i)$};
\end{scope}
\begin{scope}[xshift=140]
\draw[fill=black] (2,0) circle[ radius=0.05];
\draw[fill=black] (0,2) circle[ radius=0.05];
\draw[fill=black] (-1,0) circle[ radius=0.05];
\draw [->,  thick](0,2)--(.4,2.4);
\draw [->,  thick](2,0)--(1.6,-.4) ;
\draw [->,  thick](-1,0)--(-.6,.4);
\draw[color=black!40] (2,0)--(2,2);
\draw[color=black!40] (0,2)--(2,2);
\draw[color=black!40] (2,0)--(-1.5,0);
\draw[thick,color=black!40] (0.96,2-.13)--(.96,2+.13);
\draw[thick,color=black!40] (1.04,2-.13)--(1.04,2+.13);
\draw[thick,color=black!40] (2-.13,0.96)--(2.13,.96);
\draw[thick,color=black!40] (2-.13,1.04)--(2.13,1.04);
\draw (-.2,2) node {\small $a$};
\draw (2.2,-.2) node {\small $b$};
\draw (-1,-.2) node {\small $c$};
\draw (-1,2.5) node {\small $(ii)$};
\end{scope}
\begin{scope}[xshift=280]
\draw[fill=black] (0,3) circle[ radius=0.05];
\draw[fill=black] (0,0) circle[ radius=0.05];
\draw[fill=black] (2,3) circle[ radius=0.05];
\draw [->,  thick](0,3)--(-.4,2.6);
\draw [->,  thick](0,0)--(.4,.4) ;
\draw [->,  thick](2,3)--(2.4,3.4);
\draw[color=black!40] (0,3)--(2,3);
\draw[color=black!40] (0,0)--(0,3);
\draw (-.2,3.2) node {\small $b$};
\draw (-.2,-.2) node {\small $c$};
\draw (2.2,3) node {\small $a$};
\draw (-1,2.5) node {\small $(iii)$};
\end{scope}
\begin{scope}[xshift=420]
\draw[fill=black] (0,3) circle[ radius=0.05];
\draw[fill=black] (2,3) circle[ radius=0.05];
\draw[fill=black] (2,0) circle[ radius=0.05];
\draw [->,thick](0,3)--(.4,2.6);
\draw [->,thick](2,0)--(2.4,-.4) ;
\draw [->,thick](2,3)--(1.6,3.4);
\draw[color=black!40] (0,3)--(2,3)--(2,0);
\draw (-.2,2.8) node {\small $a$};
\draw (2.2,2.8) node {\small $b$};
\draw (1.8,-.2) node {\small $c$};
\draw (-1,2.5) node {\small $(iv)$};
\end{scope}
\begin{scope}[xshift=560]
\draw[fill=black] (0,3) circle[ radius=0.05];
\draw[fill=black] (0,0) circle[ radius=0.05];
\draw[fill=black] (2,0) circle[ radius=0.05];
\draw [->,  thick](0,3)--(.4,2.6);
\draw [->,  thick](0,0)--(-.4,.4) ;
\draw [->,  thick](2,0)--(2.4,-.4);
\draw[color=black!40] (0,3)--(0,0)--(2,0);
\draw (-.2,-.2) node {\small $b$};
\draw (-.2,3) node {\small $c$};
\draw (2,.2) node {\small $a$};
\draw (-1,2.5) node {\small $(v)$};
\end{scope}
\end{tikzpicture}
\caption{Arrow diagrams of networks on two species with capacity for MPNE. $(i):$ line $ab$ has slope -1, $c_1=b_1$, $c_2>b_2$. $(ii):$ line $ab$ has slope -1, $c_2=b_2$, $c_1<b_1$. $(iii):$ $b_1=c_1$, $a_2=b_2$, $a_1>b_1$, $c_2<b_2$. $(iv):$ $b_1=c_1$, $a_2=b_2$, $a_1<b_1$, $c_2<b_2$}\label{fig:MPNEmotifs}
\end{figure}
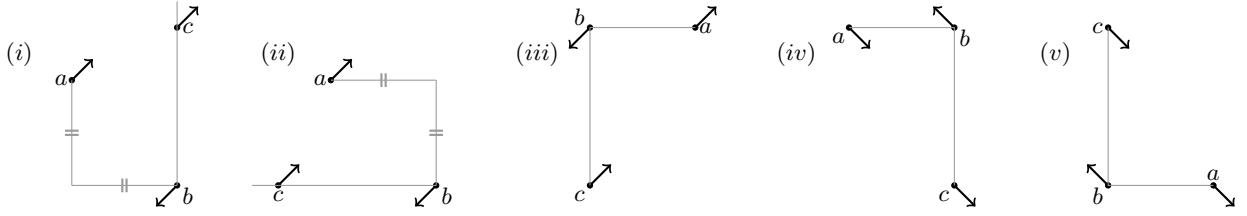

\begin{proof}
$(i)$ Let $p=a_1-b_1=b_2-a_2<0$, $q=c_2-b_2>0$ and let $\kappa_1$, $\kappa_2$, $\kappa_3$ denote the rate constants of $a\to a'$, $b\to b'$, $c\to c'$. Then 
$$
f(\eta,\kappa;t)=
(\eta_1+t)^{b_1}(\eta_2+t)^{b_2}\left(\kappa_1\left(\frac{\eta_1+t}{\eta_2+t}\right)^p
-\kappa_2
+\kappa_3(\eta_2+t)^{q}\right).
$$
Let 
$$g(t)=\kappa_1\left(\frac{\eta_1+t}{\eta_2+t}\right)^p
+\kappa_3(\eta_2+t)^{q}.$$
Note that $\eta+t_1v$, $\eta+t_2v$ are positive equilibria of $({\cal N},\kappa)$ if and only if $g(t_1)=g(t_2)>0$ (this will be the value of $\kappa_2$), and moreover they are nondegenerate if and only of $g'(t_1)$, $g'(t_2)$ are nonzero (lemma \ref{prop:DSSh}). 

We have
$$g'(t)=\kappa_1p (\eta_2-\eta_1)\frac{(\eta_1+t)^{p-1}}{(\eta_2+t)^{p+1}}+\kappa_3q(\eta_2+t)^{q-1}.$$
Set $\eta=(2,4)$ and $\kappa_1=q$, $\kappa_3=-p2^{-p-2q}$, so that $g'(0)=0$. We calculate 
$$g''(0)=pq(p-q-2)2^{-p-4}<0$$
Since $g(0)>0$ Lemma \ref{lem:ginj} applies, and we may pick $t_1\neq t_2$ close to 0 such that $g(t_1)=g(t_2)>0$ and $g'(t_1)$, $g'(t_2)$ are both nonzero. We set $\kappa_2=g(t_1)$ and lemma \ref{prop:DSSh} implies that $\eta+t_1v$, $\eta+t_2v$ are compatible nondegenerate equilibria of $({\cal N},\kappa).$

\smallskip
$(ii)$ Let $p=a_1-b_1=b_2-a_2<0$, $q=c_2-b_2<0$. The formula for $f$ and $g$ are the same as in $(i)$. Setting $\eta=(4,2)$ $\kappa_1=-q$, $\kappa_3=-p2^{p-2q}$ we obtain $g(0)=0$ and $g''(0)=(-p-q-2)2^{p-4}.$  Therefore if $p+q\neq -2$ Lemma \ref{lem:ginj} applies as above and $\cal N$ has the capacity for MPNE. If $p+q=-2$ then $p=-1$ and $q=-1$, i.e. ${\cal N}\in{\cal C}$ with corner $c$, Figure \ref{fig:C2D}(i).

\smallskip
$(iii)$ Let $p=a_1-b_1>0$ and $q=c_2-b_2<0$. We have 
$$f(\eta,\kappa;t)=(\eta_1+t)^{b_1}(\eta_2+t)^{b_2}(\kappa_1 (\eta_1+t)^p-k_2+\kappa_3 (\eta_2+t)^q)$$
and let $g(t)=\kappa_1 (\eta_1+t)^p+\kappa_3 (\eta_2+t)^q$. Set $\eta=(1,1)$ and $k_1=-q$, $k_3=p$. Then $g(0)>0$, $g'(0)=0$ and $g''(0)=pq(q-p)<0$. Therefore Lemma $\ref{lem:ginj}$ applies as above and ${\cal N}$ has the capacity for MPNE. 

\smallskip
$(iv)$ Let $p=a_1-b_1<0$ and $q=c_2-b_2<0$; We have
$$f(\eta,\kappa;t)=(\eta_1+t)^{b_1}(\eta_2-t)^{b_2}(\kappa_1 (\eta_1+t)^p-k_2+\kappa_3 (\eta_2-t)^q)$$
and let $g(t)=\kappa_1 (\eta_1+t)^p+\kappa_3 (\eta_2-t)^q.$ Set $\eta=(1,1)$ and $k_1=-q$, $k_3=-p$. Then $g(0)>0$, $g'(0)=0$ and $g''(0)=-pq(q+p-2)>0$. Therefore Lemma $\ref{lem:ginj}$ applies as above and ${\cal N}$ has the capacity for MPNE. 

\smallskip
$(v)$ Let $p=a_1-b_1<0$ and $q=c_2-b_2<0$. We have $p=a_1-b_1>0$ and $q=c_2-b_2>0$; $f$ and $g$ are as in $(iv)$. Set $\eta=(1,1)$ and $k_1=q$, $k_3=p$. Then $g(0)>0$, $g'(0)=0$ and $g''(0)=-pq(q+p-2)$. If $p+q-2\neq 0$ then ${\cal N}$ has the capacity for MPNE by Lemma $\ref{lem:ginj}$. If $p+q=2$ then $p=1$ and $q=1$ and $\{a\to a', b\to b', c\to c'\}\in{\cal C}$ with corner $b$, see Figure \ref{fig:C2D}(ii).

\end{proof}
\begin{lem}\label{lem:100}
Suppose ${\cal N}$ is a zigzag network on two species composed of three reactions whose source complexes are not collinear. Then either $\cal N$ has the capacity for MPNE or ${\cal N}\in{\cal C}$.
\end{lem}
\begin{proof}
If $\cal N$ has a zigzag of slope different from -1, then $\cal N$ has the capacity for MPNE by Theorem \ref{thm:Joshi,Shiu}. Suppose ${\cal N}=\{a\to a',b\to b', c\to c'\}$ and that $\{a\to a',b\to b'\}$ form a zigzag of slope -1 like in Figure $\ref{fig:s2z}$(iii). Suppose $c\to c'$ has the same direction as $a\to a'$ (we may permute coordinates and change the directions of all arrows to arrive at this case). Set $v=(1,1)$. If $c_1>b_1$ or $c_2<b_2$ then ${\cal N}$ is 2-alternating, and therefore has the capacity for MPNE by corollary \ref{cor:2alt}. We assume $c_1\le b_1$ and $c_2\ge b_2$.  If $c_1<b_1$ and $c_2>b_2$ then $\{b\to b', c\to c'\}$ form a zigzag of slope  different from -1, and $\cal N$ has the capacity for MPNE by Theorem \ref{thm:Joshi,Shiu}.
The remaining cases have the arrow diagrams in Figure \ref{fig:MPNEmotifs}(i),(ii) and the conclusion follows from lemma \ref{lem:39}.
\end{proof}

\begin{lem}\label{lem:02}
Let ${\cal N}$ be a rank 1 reaction network on two species composed of three reactions. 
If ${\cal N}\in$\ac$\setminus{\cal Z}$ then either $\cal N$ has the capacity for MPNE or ${\cal N}\in{\cal C}$. 
\end{lem}
\begin{proof} Let ${\cal N}=\{a\to a', b\to b', c\to c'\}$ with directions of reactions given by $\lambda_1=\lambda_3=1$ for $a\to a', c\to c'$ and $\lambda_2=-1$ for $b\to b'.$

If $v=(1,1)$ then up to changing directions of all reactions the  1D projections are as follows:
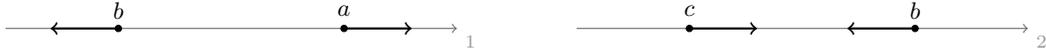
\begin{center}
\begin{tikzpicture}[scale=0.6][h]
\begin{scope}
\draw [->, black!60](0,0)--(10,0);
\draw[fill=black] (2.5,0) circle[radius=0.07];
\draw[fill=black] (7.5,0) circle[radius=0.07];
\draw [<-, thick](1,0)--(2.5,0);
\draw [->, thick](7.5,0)--(9,0);
\draw (2.5,.4) node {\small $b$}; 
\draw (7.5,.4) node {\small $a$}; 
\draw (10.3,-.3) node[color=black!40] {\scriptsize $1$}; 
\end{scope}
\begin{scope}[xshift=360]
\draw [->, black!60](0,0)--(10,0);
\draw[fill=black] (2.5,0) circle[radius=0.07];
\draw[fill=black] (7.5,0) circle[radius=0.07];
\draw [->, thick](7.5,0)--(6,0);
\draw [->, thick](2.5,0)--(4,0);
\draw (2.5,.4) node {\small $c$}; 
\draw (7.5,.4) node {\small $b$}; 
\draw (10.3,-.3) node[color=black!40] {\scriptsize $2$}; 
\end{scope}
\end{tikzpicture}
\end{center}

If $c_1<b_1$ or $a_2>b_2$ then ${\cal N}$ is 2-alternating and has the capacity for MPNE. If $c_1>b_1$ then $\{b\to b', c\to c'\}$ is a zigzag, contradiction; so $c_1=b_1$. Likewise, $a_2=b_2$. Let $p=a_1-b_1>0$ and $q=c_2-b_2<0$. The aroow diagram of $\cal N$ is that of Figure \ref{fig:MPNEmotifs}(iii). By lemma \ref{lem:39} $\cal N$ has the capacity for MPNE. 

If $v=(1,-1)$ then up to changing directions of all reactions the  1D projections are either
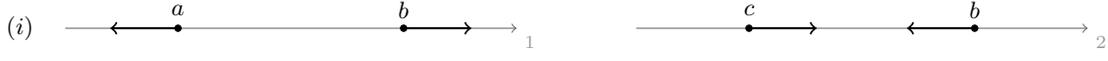
\begin{center}
\begin{tikzpicture}[scale=0.6][h]
\draw (-1,0) node {$(i)$}; 
\begin{scope}
\draw [->, black!60](0,0)--(10,0);
\draw[fill=black] (2.5,0) circle[radius=0.07];
\draw[fill=black] (7.5,0) circle[radius=0.07];
\draw [<-, thick](1,0)--(2.5,0);
\draw [->, thick](7.5,0)--(9,0);
\draw (2.5,.4) node {\small $a$}; 
\draw (7.5,.4) node {\small $b$}; 
\draw (10.3,-.3) node[color=black!40] {\scriptsize $1$}; 
\end{scope}
\begin{scope}[xshift=360]
\draw [->, black!60](0,0)--(10,0);
\draw[fill=black] (2.5,0) circle[radius=0.07];
\draw[fill=black] (7.5,0) circle[radius=0.07];
\draw [->, thick](7.5,0)--(6,0);
\draw [->, thick](2.5,0)--(4,0);
\draw (2.5,.4) node {\small $c$}; 
\draw (7.5,.4) node {\small $b$}; 
\draw (10.3,-.3) node[color=black!40] {\scriptsize $2$}; 
\end{scope}
\end{tikzpicture}
\end{center}
or 
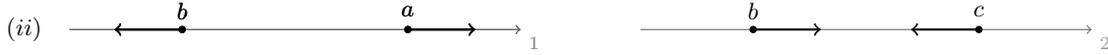
\begin{center}
\begin{tikzpicture}[scale=0.6][h]
\draw (-1,0) node {$(ii)$}; 
\begin{scope}
\draw [->, black!60](0,0)--(10,0);
\draw[fill=black] (2.5,0) circle[radius=0.07];
\draw[fill=black] (7.5,0) circle[radius=0.07];
\draw [<-, thick](1,0)--(2.5,0);
\draw [->, thick](7.5,0)--(9,0);
\draw (2.5,.4) node {\small $b$}; 
\draw (7.5,.4) node {\small $a$}; 
\draw (10.3,-.3) node[color=black!40] {\scriptsize $1$}; 
\end{scope}
\begin{scope}[xshift=360]
\draw [->, black!60](0,0)--(10,0);
\draw[fill=black] (2.5,0) circle[radius=0.07];
\draw[fill=black] (7.5,0) circle[radius=0.07];
\draw [->, thick](7.5,0)--(6,0);
\draw [->, thick](2.5,0)--(4,0);
\draw (2.5,.4) node {\small $b$}; 
\draw (7.5,.4) node {\small $c$}; 
\draw (10.3,-.3) node[color=black!40] {\scriptsize $2$}; 
\end{scope}
\begin{scope}
\draw [->, black!60](0,0)--(10,0);
\draw[fill=black] (2.5,0) circle[radius=0.07];
\draw[fill=black] (7.5,0) circle[radius=0.07];
\draw [<-, thick](1,0)--(2.5,0);
\draw [->, thick](7.5,0)--(9,0);
\draw (2.5,.4) node {\small $b$}; 
\draw (7.5,.4) node {\small $a$}; 
\draw (10.3,-.3) node[color=black!40] {\scriptsize $1$}; 
\end{scope}
\end{tikzpicture}
\end{center}
First consider case $(i)$ in the diagram. If $c_1>b_1$ or $a_2>b_2$ then ${\cal N}$ is 2-alternating and has the capacity for MPNE. If $c_1<b_1$ then $\{b\to b', c\to c'\}$ is a zigzag, contradiction; so $c_1=b_1$. Likewise, $a_2=b_2$. The arrow diagram of $\cal N$ is that of Figure \ref{fig:MPNEmotifs}(iv), and $\cal N$ has the capacity for MPNE by lemma \ref{lem:39}. 

For case (ii) it is argued similarly that $c_1=b_1$ and $a_2=b_2$. This corresponds to the arrow diagram in Figure \ref{fig:MPNEmotifs}(v) and by lemma \ref{lem:39} either $\cal N$ has the capacity for MPNE of ${\cal N}\in{\cal C}$.  
\end{proof}

\medskip

For the proof of theorem \ref{thm:mainext} (3) for $n=2$ we consider cases according to the minimum size of all \ac subnetworks of $\cal N$: two (i.e. ${\cal N}\in{\cal Z}$) or three reactions (by Proposition \ref{prop:s2nz} \ac networks that require four reactions to form the $(\to, \gets)$ and $(\gets, \to)$ patterns are in ${\cal S}_2^{nz}$). The following two propositions combine to give a proof of Theorem \ref{thm:mainext}(3) for networks on two species. 

\begin{prop}
Let $\cal N$ be a 1-alt complete network on two species. If ${\cal N}\in {\cal Z}\setminus ({\cal S}_2^z\cup{\cal L}\cup{\cal C})$ then $\cal N$ has the capacity for MPNE. 
\end{prop}
\begin{proof} Suppose $\cal N$ does not have the capacity for MPNE. 

Let ${a\to a',b\to b'}$ form a zigzag (necessarily of slope -1) like the one in Figure $\ref{fig:s2z}$(iii). Let $c\to c'$ be a different reaction of $\cal N$. 
By Lemma \ref{lem:100} either $a,b,c$ are collinear or ${a\to a',b\to b', c\to c'}\in{\cal C}$, in other words $c=(a_1,b_2)$. Suppose we have both types, i.e.  $c\to c'$ and $d\to d'$ are such that ${a\to a',b\to b', c\to c'}\in{\cal C}$ and $a,b,d$ are collinear. Let $\kappa_1,\kappa_2,\kappa_3,\kappa_4$ denote the rate constants of $a\to a'$, $b\to b'$, $c\to c'$, $d\to d'$. We have $d_1+d_2=b_1+b_2$, so $d_2-c_2=1-(d_1-c_1)=1-p$. We have 
\begin{eqnarray*}
f(\eta,\kappa;t)&=&
(\eta_1+t)^{c_1}(\eta_2+t)^{c_2}\left(\kappa_1(\eta_2+t)-\kappa_2(\eta_1+t)
+\kappa_3
+\kappa_4\lambda_4(\eta_1+t)^p(\eta_2+t)^{1-p}
\right)\\
&=&(\eta_1+t)^{c_1}(\eta_2+t)^{c_2}(\kappa_3-g(t))
\end{eqnarray*}
where $g(t)=-\kappa_1(\eta_2+t)+\kappa_2(\eta_1+t)
-\kappa_4\lambda_4(\eta_1+t)^p(\eta_2+t)^{1-p}$ and $\lambda_4=1$ or $\lambda_4=-1$.
Note that $p=0$ or $p=1$ give $d=a$ or $d=b$, so we have $p<0$ or $p>1$. Suppose $p>1$ (the other case is similar). If $\lambda_4=1$, we set $\eta=(2,4)$ and we let $\kappa_1=2^{-p-1}(p-1)$, $\kappa_2=2^{-p-1}(3p+1)$, $\kappa_4=1$ and we get $g(0)=2^{-p+2}p$, $g'(0)=0$ and $g''(0)=-2^{-p-1}p(p-1)\neq 0$. If $\lambda_4<0$ we set $\eta=(4,2)$, $\kappa_1=2^p$, $\kappa_2=2^{p-1}p$ and $\kappa_4=1$ and we have $g(0)=2^{p+1}p$, $g'(0)=0$ and $g''(0)=2^{p-3}p(p-1)\neq 0$. In both cases 
Lemma \ref{lem:ginj} applies as above and $({\cal N},\kappa)$ has multiple nondegenerate equilibria, contradiction. 

Therefore either ${\cal N}\in{\cal C}$ or ${\cal N}$ has all source complexes on a line of slope -1. In the latter case if ${\cal N}\notin {\cal S}_2^z\cup{\cal L}$ then $\cal N$ is 2-alternating and has the capacity for MPNE, contradiction. 
\end{proof}

\begin{prop}
Let $\cal N$ be a 1-alt complete network on two species. Suppose  ${\cal N}\notin ({\cal Z}\cup {\cal C})$ and that $\cal N$ has an 1-alt complete subnetwork composed of three reactions. Then $\cal N$ has the capacity for MPNE. 
\end{prop}
\begin{proof}
Suppose the subnetwork $\{a\to a', b\to b', c\to c'\}$ of $\cal N$ is \ac with $a\to a', c\to c'$ having the same direction. By Lemma \ref{lem:02} $\cal N$ has the capacity for MPNE unless the arrow diagram of $\{a\to a', b\to b', c\to c'\}$ is that in Figure \ref{fig:C2D}(ii). Assume that's the case, and let $d\to d'$ be a reaction of $\cal N$ with $d\notin\{a,b,c\}$. 

First assume $d\to d'$ has the same direction as $a\to a'$. If $d_1>b_1$ then $\{a\to a', b\to b', d\to d'\}\in$\ac$\setminus{\cal Z}$ and since $d\neq c$ by Lemma \ref{lem:02} $\cal N$ has the capacity for MPNE. If $d_1<b_1$ then $\{b\to b', c\to c', d\to d'\}\in$\ac$\setminus{\cal Z}$ and by Lemma \ref{lem:02} $\cal N$ has the capacity for MPNE. Likewise, $\cal N$ has the capacity for MPNE if $d_2\neq b_2$. Since $d\neq b$ we conclude that $\cal N$ has MNPE. 

Assume $d\to d'$ has the  direction of $b\to b'$. If $d_1>c_1$ or $d_2>a_2$ then ${\cal N}\in$\at has capacity for MPNE. Assume $d_1\le c_1$ and $d_2\le a_2$. If $d_1<a_1$ and $d_2<a_2$ then $\{a\to a', d\to d'\}$ is a zigzag, contradiction. Likewise, if $d_1<c_1$ and $d_2<c_2$ $\{c\to c', d\to d'\}$ is a zigzag, contradiction. We therefore have $d_1\in\{a_1,c_1\}$, $d_2\in\{a_2, c_2\}$. Since $d$ is different from $a,b,c$ we have $d=(c_1, a_2)$. Then $\{a\to a', d\to d', c\to c'\}\in$\ac$\setminus{\cal C}$, and by lemma \ref{lem:02} $\cal N$ has the capacity for MPNE. 
\end{proof}

\subsubsection{Networks on $n\ge 3$ species} 
{\em Let $\cal N$ be an essential 1-alt complete reaction network on $n\ge 2$ species, and suppose $\cal N$ does not have the capacity for MPNE.} We prove that ${\cal N}\in {\cal D}$ by using inheritance (Theorems \ref{thm:inheritance}, \ref{thm:inheritance2}; see also the remarks at the beginning of section \ref{sec:1Drn}). The outline of the argument is this: $\cal N$ does not have the capacity for MPNE, and therefore neither do any of its projections. In particular, all \ac 2D projections of $\cal N$ are in $\cal D$ by case $n=2$. Noting that all \ac networks have 2D projections that are \ac, the following propositions combine to show that $\cal N$ itself is in $\cal D$, completing the proof of Theorem \ref{thm:mainext}. The assumptions on $\cal N$ stated above carry over in all three statements below. 

\begin{prop}\label{lem:S2zLproj}
If $\cal N$ has a 2D projection in ${\cal S}_2^{z}\cup {\cal L}$, then ${\cal N}\in {\cal S}_2^{z}\cup {\cal L}$.
\end{prop}
\begin{proof}
Suppose ${\cal N}_{\{1,2\}}\in {\cal S}_2^{z}\cup \cal L$, and assume that reactions $a\to a'$ and $b\to b'$ of $\cal N$ form a zigzag in ${\cal N}_{\{1,2\}}$. Since $\cal N$ does not have the capacity for MPNE it follows from corollary \ref{cor:no2zz2} that $\{1,2\}$ is the only projection of $\{a\to a', b\to b'\}$ with zigzag, and therefore we have $a_k=b_k$ for all $k\in S\setminus \{1,2\}$ by lemma \ref{cor:31}.
Any reaction $c\to c'$ with $c\neq a$ and having the same direction as $b\to b'$ can replace $b\to b'$ in the argument above, and therefore $a_k=c_k$ for all $k\in S\setminus \{1,2\}$. In the same way $c_k=b_k$ for all $k\in S\setminus \{1,2\}$ when $c\to c'$ has the same direction as $a\to a'$ and $c\neq b$. Therefore all 1D projections other than on $\{1\}$, $\{2\}$ have a single source complex, which contradicts the fact that $\cal N$ is an essential network if $n>3$. Therefore $n=2$ and  ${\cal N}={\cal N}_{\{1,2\}}$.
\end{proof}

\begin{prop}\label{prop:main}
If $\cal N$ has a 2D projection in ${\cal C}$ then ${\cal N}\in {\cal C}$.
\end{prop}
\begin{proof} 
First we consider networks on three species, $n=3$. Let ${\cal N}=\{a\to a',b\to b', c\to c'\}$ and suppose that ${\cal N}_{\{1,2\}}$ forms one of the  arrow diagram  in Figure \ref{fig:C2D}.

{\em Case 1: ${\cal N}_{\{1,2\}}\in{\cal Z}$, Figure \ref{fig:C2D}(i)}. If $a_3\neq c_3$ then $\{a\to a', b\to b'\}$ is a zigzag on the $\{1,2\}$ or $\{1,3\}$ projections; this subnetwork of $\cal N$ also forms a zigzag on the $\{1,2\}$ projection; by corollary \ref{cor:no2zz2} $\cal N$ has the capacity for MPNE, contradiction. Therefore $a_3=b_3$. If $a_3=b_3=c_3$ then since $\{a\to a', b\to b'\}$ are arbitrary representatives of all reactions with the same projections on $\{1,2\}$ it follows that ${\cal N}_{\{k\}}$ has a single source complex. But $\cal N$ is essential, contradiction.

We assume that $c_3\neq a_3=b_3$ and suppose $a\to a'$ has positive direction in ${\cal N}_{\{3\}}$, i.e. $a'_3>a_3$ (the case of the opposite direction is treated in the same way). 
If $c_3<a_3$ 
then ${\cal N}_{\{1,3\}}$ has the arrow diagram in Figure \ref{fig:casesforC}(i), and so ${\cal N}\in$\ac$\setminus {\cal D}$, contradicting the lack of capacity for MPNE of $\cal N$. 

Therefore $a_3=b_3<c_3$, and it follows that ${\cal N}_{\{1,3\}}$ has the arrow diagram in Figure \ref{fig:casesforC}(ii). But ${\cal N}_{\{1,3\}}\in {\cal D}$ so we must have ${\cal N}\in {\cal C}$, and $c_k-a_k=1$. Letting $\gamma=(a_1,b_2,a_1)$ we have $a-\gamma=(0,1,0)$, $b-\gamma=(1,0,0)$ and $c-\gamma =(0,1,0)$. Therefore ${\cal N}\in {\cal C}$.

\smallskip

{\em Case 2: ${\cal N}_{\{1,2\}}\in{\cal C}\setminus{\cal Z}$, Figure \ref{fig:C2D}(ii)}. If $\cal N$ has a 2D projection in ${\cal Z}\cap {\cal C}$ then Case 1 implies that ${\cal N}\in {\cal C}$. We show that otherwise we get contradiction.
Suppose no 2D projection of ${\cal N}$ is in $(\cal C)\cup{\cal Z}$. We have $$v_1(c_1-b_1)>0 \text{ and }v_2(b_2-a_2)>0$$ and therefore 
$$v_3(c_3-b_3)\ge 0\text{ and }v_3(b_3-a_3)\ge 0,$$ 
in particular $(c_3-b_3)(b_3-a_3)\ge 0$, i.e. either $(c_3-b_3)(b_3-a_3)>0$ or ${\cal N}_{\{3\}}\in$\at. The latter contradicts the lack of capacity for MPNE of $\cal N$.

Assume $a_3=c_3$ (the other case is similar). If $a_3=b_3=c_3$ then as above ${\cal N}_{\{3\}}$ contains a single source complex and $\cal C$ is not essential, contradiction. If $a_3=c_3\neq b_3$ then either $\{b\to b', c\to c'\}$ is a zigzag on the $\{1,3\}$ projection, or $\{a\to a', b\to b'\}$ is a zigzag on the $\{2,3\}$ projection. The projections of $a$, $b$, $c$ on $\{1,3\}$ are pairwise distinct: for example $a_3\neq b_3$, $b_3\neq c_3$, and $c_1\neq a_1$. Moreover, they are not collinear. Since ${\cal N}\in{\cal D}$ it follows that ${\cal N}_{\{1,3\}}\in {\cal C}\cap{\cal Z}$, contradicting our assumption. 

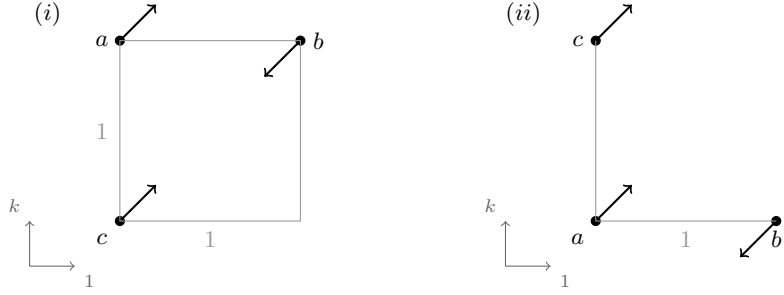
\begin{figure}[t]
\centering
\begin{tikzpicture}[scale=1.2]
\begin{scope}
\draw [->, color=black!60](-1,-.5)--(-.5,-.5) node[anchor=north west]{\scriptsize $1$};
\draw [->, color=black!60](-1,-.5)--(-1,0) node[anchor=south east]{\scriptsize $k$};
\draw[fill=black] (2,2) circle[ radius=0.05];
\draw[fill=black] (0,2) circle[ radius=0.05];
\draw[fill=black] (0,0) circle[ radius=0.05];
\draw [->,  thick](0,2)--(.4,2.4);
\draw [->,  thick](2,2)--(1.6,1.6);
\draw [->,  thick](0,0)--(.4,.4);
\draw [color=black!40] (2,0)--(0,0)--(0,2)--(2,2)--(2,0);
\draw[color=black!40] (1,-.2) node {\small $1$}; 
\draw[color=black!40] (-.2,1) node {\small $1$}; 
\draw (-.2,2) node {\small $a$};
\draw (2.2,2) node {\small $b$};
\draw (-.2,-.2) node {\small $c$};
\draw (-.8,2.3) node {\small $(i)$};
\end{scope}
\begin{scope}[xshift=150]
\draw [->, color=black!60](-1,-.5)--(-.5,-.5) node[anchor=north west]{\scriptsize $1$};
\draw [->, color=black!60](-1,-.5)--(-1,0) node[anchor=south east]{\scriptsize $k$};
\draw[fill=black] (2,0) circle[ radius=0.05];
\draw[fill=black] (0,2) circle[ radius=0.05];
\draw[fill=black] (0,0) circle[ radius=0.05];
\draw [->,  thick](0,2)--(.4,2.4);
\draw [->,  thick](2,0)--(1.6,-.4) ;
\draw [->,  thick](0,0)--(.4,.4);
\draw [color=black!40] (2,0)--(0,0)--(0,2);
\draw[color=black!40] (1,-.2) node {\small $1$}; 
\draw (-.2,2) node {\small $c$};
\draw (2,-.2) node {\small $b$};
\draw (-.2,-.2) node {\small $a$};
\draw (-.8,2.3) node {\small $(ii)$};
\end{scope}
\end{tikzpicture}
\caption{Cases for the proof of Proposition  \ref{prop:main}}\label{fig:casesforC}
\end{figure}

\medskip

Now we consider $n\ge 4$. 
Assume by permuting coordinates if necessary that ${\cal N}_{\{1,2\}}\in{\cal C}$ has one of the arrow diagrams in Figure \ref{fig:C2D}. For each $k\in S$, ${\cal N}_{\{k\}}$ contains at least two source complexes ($\cal N$ is essential).  Let $\gamma_k$ denote the minimum of all source complexes of ${\cal N}_{\{k\}}$. For any $k\ge 3$, ${\cal N}_{\{1,2,k\}}\in{\cal D}$ is an essential 1-alt complete network on three species, and so it must be in $\cal C$. It follows that for any complex $d$ of $\cal N$ and for any $k\in S$  $d_k=\gamma_k$ or $d_k=\gamma_k+1$. We check that the latter cannot hold on more than one coordinate $k$. Suppose $k,l$ are different and
$d_k=\gamma_k$, $d_l=\gamma_l$. Each 1D projection of $\cal N$ contains at least one of the $(\to,\gets)$ or $(\gets, \to)$ patterns, and so either ${\cal N}_{\{1,k,l\}}$ is 1-alt complete or ${\cal N}_{\{2,k,l\}}$ is. In either case, that 3D projection of $\cal N$ must be in $\cal C$ by hypothesis and therefore the projection of $d-\gamma$ on $\{1,k,l\}$ has only one nonzero entry, contradiction. 
Therefore for any source complex $d$ of $\cal N$ we have $d-\gamma\in\{e_1,\ldots, e_n\}$ and ${\cal N}\in{\cal C}$.
\end{proof}

\begin{prop}\label{lem:S2nzproj}
If $\cal N$ has a 2D projection in ${\cal S}_2^{nz}$, then ${\cal N}\in {\cal S}_2^{nz}$.
\end{prop}
\begin{proof} Let ${\cal N}_{\{1,2\}}=\{a\to a', b\to b', c\to c', d\to d'\}$  have one of the arrow diagrams in 
Figure \ref{fig:S2nz}. Let $k\neq 1,2$. If $a_k\neq b_k$ then  ${\cal N}_{\{1,k\}}\in$\ac and it contains at least three source complexes: $(a_1,a_k),(a_1,b_k),(c_1,c_k)$ are distinct and moreover not collinear, which implies that ${\cal N}_{\{1,k\}}\in{\cal C}$, and by Proposition \ref{prop:main} ${\cal N}\in{\cal C}$. However, $\cal C$ all 2D projections of $\cal C$ networks have three source complexes, contradicting ${\cal N}_{\{1,2\}}\in{\cal S}_2^{nz}$. Likewise, if $c_k\neq d_k$ we get a contradiction.

It remains to consider the case $a_k=b_k$ and $c_k=d_k$. If $a_k=b_k=c_k=d_k$ then since $a\to a', b\to b', c\to c', d\to d'$ were arbitrary representatives of reactions with the same projections on $\{1,2\}$, it follows that ${\cal N}_{\{k\}}$ has a single source complex, contradiction. We must have $n=2$ and ${\cal N}={\cal N}_{\{1,2\}}\in{\cal S}_2^{nz}$. If on the other hand $a_k=b_k\neq c_k=d_k$ for all $k$ then ${\cal N}$ has two source complexes $a=b$ and $c=d$ and either $v_1v_k(c_1-a_1)(c_k-a_k)$ has the same sign for all $k\neq 1$ (i.e. ${\cal N}\in {\cal S}_2^{nz}$) or for some $k\neq 1,2$ the projection ${\cal N}_{\{1,k\}}$ is a zigzag. In the latter case, since ${\cal N}_{\{1,k\}}\in\cal D$ we must have ${\cal N}={\cal N}_{\{1,k\}}\in{\cal S}_2^z$, contradiction. 
\end{proof}

\bibliographystyle{siam}    
\bibliography{rank1Multistat}
\end{document}